\documentclass{amsart}
\usepackage{amsfonts,amsthm,amsmath}
\usepackage{amssymb}
\usepackage{amscd}
\usepackage[utf8]{inputenc}
\usepackage[a4paper]{geometry}
\usepackage{enumerate}
\usepackage[usenames,dvipsnames]{color}
\usepackage{array}
\usepackage{array,tabularx}

\numberwithin{equation}{section}
\numberwithin{equation}{subsection}

\theoremstyle{plain}

\newtheorem{theorem}[equation]{Theorem}
\newtheorem{lemma}[equation]{Lemma}
\newtheorem{proposition}[equation]{Proposition}

\newtheorem{corollary}[equation]{Corollary}

\newtheorem{cor}[equation]{Corollary}

\theoremstyle{definition}
\newtheorem{example}[equation]{Example}
\newtheorem{remark}[equation]{Remark}

\newtheorem{definition}[equation]{Definition}

\newtheorem{bekezdes}[equation]{}

\def\C{\mathbb C}
\def\Q{\mathbb Q}

\def\Z{\mathbb Z}

\def\im{{\rm Im}}

\newcommand{\calv}{{\mathcal V}}

\newcommand{\calm}{{\mathcal M}}

\newcommand{\calG}{{\mathcal G}}

\newcommand{\calO}{{\mathcal O}}

\newcommand{\calS}{{\mathcal S}}

\newcommand{\calL}{\mathcal{L}}

\newcommand{\tX}{\widetilde{X}}

\newcommand{\cO}{{\mathcal O}}

\newcommand{\bC}{{\mathbb C}}

\newcommand{\eca}{{\rm ECa}}
\newcommand{\pic}{{\rm Pic}}

\newcommand{\bZ}{{\mathbb{Z}}}
\newcommand{\bQ}{{\mathbb{Q}}}

\author{J\'anos Nagy}
\address{Central European University, Dept. of Mathematics,  Budapest, Hungary}
\email{nagy\textunderscore janos@phd.ceu.edu}

\author{Andr\'as N\'emethi}
\address{Alfr\'ed R\'enyi Institute of Mathematics,
Hungarian Academy of Sciences,
Re\'altanoda utca 13-15, H-1053, Budapest, Hungary \newline
 \hspace*{4mm} ELTE - University of Budapest, Dept. of Geometry, Budapest, Hungary \newline \hspace*{4mm}
BCAM - Basque Center for Applied Math.,
Mazarredo, 14 E48009 Bilbao, Basque Country – Spain}
\email{nemethi.andras@renyi.mta.hu }

\title{The dimension of the image of the Abel map \\
associated with normal
 surface singularities
}

\begin{document}

\keywords{normal surface singularities, links of singularities,
plumbing graphs, rational homology spheres,
Abel map, effective Cartier divisors, Picard group, Brill--Noether theory,
Laufer duality, surgery formulae,
superisolated singularities, cohomology of line bundles}
\thanks{The authors are partially supported by
 NKFIH Grant ``\'Elvonal (Frontier)'' KKP 126683.}
\subjclass[2010]{Primary. 32S05, 32S25, 32S50, 57M27
Secondary. 14Bxx, 14J80, 57R57}

\begin{abstract}
Let $(X,o)$ be a complex normal surface singularity with rational homology sphere link
and let $\tX$ be one of its good resolutions. Fix an effective cycle $Z$
supported on the exceptional curve and also a possible Chern class $l'\in H^2(\tX,\Z)$.
Define ${\rm Eca}^{l'}(Z)$ as the space of effective Cartier divisors on $Z$ and
$c^{l'}(Z):{\rm Eca}^{l'}(Z)\to {\rm Pic}^{l'}(Z)$, the corresponding Abel map. In this note we provide two algorithms, which provide the dimension of the image of the Abel map.

Usually, $\dim {\rm Pic}^{l'}(Z)=p_g$, $\dim\,\im (c^{l'}(Z))$ and
  ${\rm codim}\,\im (c^{l'}(Z))$ are not topological, they are in subtle relationship with
  cohomologies of certain line bundles. However, we provide combinatorial formulae for them
  whenever the analytic structure on $\tX$ is generic.

  The  ${\rm codim}\,\im (c^{l'}(Z))$  is related with
  $\{h^1(\tX,\calL)\}_{\calL\in \im (c^{l'}(Z))}$; in order to treat the `twisted'
  family  $\{h^1(\tX,\calL_0\otimes \calL)\}_{\calL\in \im (c^{l'}(Z))}$ we need to
  elaborate a generalization of the Picard group and of the Abel map.
  The above algorithms are also generalized.
\end{abstract}

\maketitle

\linespread{1.2}


\pagestyle{myheadings} \markboth{{\normalsize  J. Nagy, A. N\'emethi}} {{\normalsize The image of the Abel map }}


\section{Introduction}\label{s:intr}
\subsection{}
Fix a complex normal surface singularity $(X,o)$ and let $\tX$ be one of its good resolutions.
 We assume that the link of $(X,o)$ is a rational homology sphere.
Denote by  $L$ the lattice $H_2(\tX,\Z)$ (endowed with its negative definite intersection form),
by  $L'$   its dual lattice $H^2(\tX,\Z)$ and by $\calS'\subset L'$
the Lipman cone of antinef cycles.
The irreducible exceptional curves are denoted by $\{E_v\}_{v\in\calv}$, their duals in $L'$ by
$\{E_v^*\}_{v\in\calv}$, $E:=\cup_vE_v$.
(For details see section \ref{s:prel}).

In \cite{NNI} for any effective cycle $Z\geq E$ and Chern class $l'\in-\calS'$
the authors introduced (based on \cite{Groth62,Kl,Kleiman2013})
and investigated the set of effective Cartier divisors $\eca^{l'}(Z)$ and the corresponding Abel maps $c^{l'}(Z) : \eca^{l'}(Z) \to \pic^{l'}(Z)$, where $\pic^{l'}(Z)$
is the affine subspace of the Picard group  of line bundles over $Z$ with Chern class $l'$.
The image of the Abel map consists of line bundles without fixed components.
\cite{NNI} and follow-up articles contain
 several properties of the Abel map, e.g.   characterisation
when it is dominant, or its relationship with cohomological properties of line bundles.
See \cite{NNII} and \cite{NNIII} for the study in the case  of generic and  elliptic singularities.
In all these treatments the investigation of the image $\im(c^{l'}(Z))$ was extremely useful.
The main goal of the present article is the computation of $\dim\,\im(c^{l'}(Z))$ and the deduction  of several new consequences.
We consider these as  necessary  steps towards a long--term final goal:
the development of the   Brill--Noether theory of normal surface singularities.

Though the dimension $(l',Z)$ (and the homotopy type) of the connected
complex manifold  $\eca^{l'}(Z)$ is topological (i.e. it depends only on the link, or on
the lattice $L$), the dimension $h^1(\calO_Z)$
of the target affine space $\pic^{l'} (Z)$ depends essentially on the analytic structure:
if we fix the topological type (and $Z$), the cohomology group $H^1(\calO_Z)$ usually depends on
the chosen analytic structure supported by the fixed topological type. The same is true
for both $\dim\, \im (c^{l'}(Z))$ and ${\rm codim}\, \im (c^{l'}(Z))$: though (surprisingly)
there is a topological characterisation of those cases when $c^{l'}(Z))$ is dominant, oppositely,
 the cases e.g. when $c^{l'}(Z))$ is a point or it is a hypersurface have no such topological characterisations.
In particular, both integers $\dim\, \im (c^{l'}(Z))$ and ${\rm codim}\, \im (c^{l'}(Z))$
are subtle analytical invariants. In fact, it turns out that
 ${\rm codim}\, \im (c^{l'}(Z))$ equals $h^1(Z,\calL_{gen}^{im})$, where $\calL_{gen}^{im}$
 is a generic line bundle from $\im (c^{l'}(Z))$. For more about such  general
 facts regarding the Abel maps (and also  about several concrete examples) see \cite{NNI,NNII,NNIII}.

Maybe it is  worth to emphasize that in the case of the Abel map associated with a smooth projective curve the dimension of the image is immediate (for this classical case consult e.g.
\cite{ACGH,Flamini}). This (and almost any other comparison)  shows the huge technical differences
between the classical smooth curve cases and our situation (which, basically, is the Brill--Noether theory of a non--reduced exceptional curve supported by the exceptional set
 of  a surface singularity resolution).
\subsection{The algorithms}
In the body of the article we present two inductive algorithm for the computation of
$d_Z(l'):=\dim\, \im (c^{l'}(Z))$. The induction follows a sequential blow up procedure
starting from the resolution $\tX$.
Write  $-l' = \sum_{v \in \calv} a_v E_v^*\in\calS'\setminus \{0\}$ (hence each $a_v\in\Z_{\geq 0}$).
Then, for every $v\in \calv$ with $a_v>0$ we fix $a_v$ generic points on $E_v$, say
$p_{v,k_v}$, $1\leq k_v\leq a_v$. Starting from each $p_{v,k_v}$ we consider a sequence of blowing ups:
first we blow up $p_{v,k_v}$ and we create the exceptional curve $F_{v,k_v,1}$,
then we blow up a generic point of
$F_{v,k_v,1}$ and we create $F_{v,k_v,2}$, and we do this, say, ${\bf s}_{v,k_v}$ times
(an exact bound is given in \ref{bek:algsetup}).
We proceed in this way with all points $p_{v,k_v}$, hence we get $\sum_va_v$ chains of modifications.
Hence, a set of integers
${\bf s}=\{{\bf s}_{v,k_v}\}_{v\in\calv,\ 1\leq k_v\leq a_v}$  provides a
modification  $\pi_{{\bf s}}:\tX_{{\bf s}}\to \tX$.
In $\tX_{{\bf s}}$ we find the exceptional curves $\cup_{v\in\calv}E_v\cup \cup_{v,k_v}
\cup_{1\leq t \leq {\bf s}_{v,k_v}}F_{v,k_v,t}$.
At each level ${\bf s}$ we set $Z_{{\bf s}}:=\pi_{{\bf s}}^*(Z)$ and
$-l'_{{\bf s}}:= \sum_{v,k_v}F^*_{v,k_v,{\bf s}_{v,k_v}}$
(in $L'(\tX_{{\bf s}})$, where
 $F_{v,k_v,0}=E_v$).
  We also write
  $d_{{\bf s}}:=\dim \im (c^{l'_{{\bf s}}}(Z_{{\bf s}}))$. Note that
  $d_{{\bf 0}}=d_Z(l')$, and it turns out that $d_{{\bf s}}=0$ whenever the entries of ${\bf s}$ are
  large enough. (Sometimes we abridge the pair $(v,k_v)$ by $(v,k)$.)

In order to run an induction,
 for any ${\bf s}$ and $(v,k)$ let ${\bf s}^{v,k}$ denote that  tuple which
is obtained from
${\bf s}$ by increasing ${\bf s}_{v,k}$ by one. The inductive algorithm compares
$d_{{\bf s}}$ with all possible $  d_{{\bf s}^{v,k}}$.

Using the fact (cf. the proof of Theorem \ref{th:ALGORITHML}) that
$\eca^{l'_{{\bf s}^{v,k}}}(Z_{{\bf s}^{v,k}})$ is birational with a codimension one subspace of
$\eca^{l'_{{\bf s}}}(Z_{{\bf s}})$, we obtain
\begin{equation}\label{eq:jumpint}
d_{{\bf s}}- d_{{\bf s}^{v,k}}\in\{0,1\}.\end{equation}
A very  subtle part of the theory is to identify all those pairs
$({\bf s}, {\bf s}^{v,k})$, where the gaps/jumps occur (that is, when the difference in (\ref{eq:jumpint})
is 0 or 1).
The identification of such places carries  a deep analytic content (and even if in some cases it
can be characterised topologically --- e.g., in the case of a generic analytic structure ---,
it might be guided by rather complicated combinatorial patterns).

 \begin{example} To create a good intuition for such a phenomenon, let us recall the classical
 case of Weierstrass points. Let $C$ be a smooth projective complex curve of genus $g$
 and let us fix a
 point $p\in C$. For any $s\in\Z_{\geq 0}$ consider $\ell(s):=h^0(C, \calO_C(sp))$.
 Then $\ell(0)=1$ and $\ell(2g-1+k)=g+k$ for $k\geq 0$. Moreover, $\ell(s)-\ell(s-1)\in \{0,1\}$
 for any $s\geq 0$. Those $s$ values when this difference is 0 are called the gaps,
  there are $g$ of them. For a generic point the gaps are $\{1, 2,\ldots, g\}$,
  otherwise
 $p$ is called a   Weierstrass point. For  Weierstrass point the set of gaps might depend on the choice of $p$  and on the analytic structure of $C$. The characterization of all
 possible gap--sets is still unsettled.
 \end{example}

In order to characterize completely our  gaps/jump places, we will use {\it test functions}.
For such a test function, say $\tau_{{\bf s}}$, we will require  the following properties.
Firstly, it is a function ${\bf s}\mapsto \tau_{{\bf s}}\in \Z_{\geq 0}$, such that
$d_{{\bf s}}\leq \tau_{{\bf s}}$  for any ${\bf s}$. Usually, $\tau_{{\bf s}}$ is defined by
a weaker (more robust) geometric construction, which approximates/bounds $\im (c^{l'}(Z))$,
and which hopefully is easier to compute. Secondly, $t_{{\bf s}}$ satisfies the following remarkable
{\it testing property} formulated by the next pattern theorem.

\vspace{1mm}

\noindent {\bf Pattern Theorem. } {\it The sequence of integers $d_{{\bf s}}$ are
determined inductively as follows:

(1) $d_{{\bf s}}- d_{{\bf s}^{v,k}}\in\{0,1\}$ (cf. (\ref{eq:jumpint})),

(2) if for some fixed ${\bf s}$ the numbers $\{d_{{\bf s}^{v, k}}\}_{v,k}$ are not the same,
then $d_{{\bf s}} = \max_{v, k}\{\,d_{{\bf s}^{v, k}}\}$.
 In the case when all the numbers $\{d_{{\bf s}^{v, k}}\}_{v,k}$ are the same,
 then if this common value $d_{{\bf s}^{v, k}}$ equals $\tau_{{\bf s}}$, then $d_{{\bf s}} =
\tau_{{\bf s}}  =d_{{\bf s}^{v, k}}$;  otherwise $d_{{\bf s}} = d_{{\bf s}^{v, k}}+1$.}

\vspace{1mm}

More precisely, we wish to determine from the collection $\{d_{{\bf s}^{v, k}}\}_{v,k}$
the term $d_{{\bf s}}$ (as a decreasing induction). Using {\it (1)} this is ambiguous only if all this numbers are the same, say $d$. In this case $d_{{\bf s}}$ can be $d$ or $d+1$. Well,
if the inequality ($\dag$) $d_{{\bf s}}\leq \tau_{{\bf s}}$  is not obstructed by the choice of
 $d_{{\bf s}}=d+1$, then this value is taken.  Otherwise it is $d$. That is, $d_{{\bf s}}$
 is as large as it can be,  modulo {\it (1)} and ($\dag$).

This can be an interesting procedure even if $s$ is a 1--entry parameter. E.g., in the case of
classical Weierstrass points, the inequality $\ell(s)\leq 1+\lfloor s/2\rfloor$ (valid for
$s\leq 2g-1$),
given by Clifford's theorem,  by this `maximal--testing procedure' gives the
sequence $\{1,1,2,2,\dots\}$ for $s\geq 0$, with gaps $\{1,3, \ldots, 2g-1\}$.
In fact,
in the case of hyperelliptic curves the  Weierstrass points are the branch points of the hyperelliptic projection and their gap--set is uniformly $\{1,3,5,\ldots, 2g-1\}$.
(However, for non--hyperelliptic curves we are not aware of the existence of a non-trivial
 test function.)

If the Pattern Theorem from above holds, then it turns out (see e.g. Corollary \ref{th:ALGORITHM3})
that
$d_{{\bf s}}=\min_{{\bf s}\leq \widetilde{{\bf s}}} \{ |\widetilde{{\bf s}}-
{\bf s}|+\tau_{\widetilde{\bf s}}\}$  for any ${\bf s}$.
(Here $|{\bf s}|=\sum_{v,k} s_{v,k_v}$.)
In particular,
\begin{equation}\label{eq:maxtype}
d_Z(l')=d_{{\bf 0}}=\min_{{\bf 0}\leq {\bf s}} \{ |{\bf s}|+\tau_{{\bf s}}\}.\end{equation}
Such type of formulas already appeared in the computation of $d_Z(l') $ for weighted homogeneous singularities (and specific $l'$) in \cite{NNI}, case which lead us to the present general case.
(The type of formula, and also the conceptual approach behind,  can also be compared e.g. with Pflueger's formula
regarding the dimension of the Brill--Noether varieties  of a generic
smooth projective curve $C$ with fixed gonality, cf. \cite{Pfl,JR}.)
Nevertheless, the approach of the testing function (and the corresponding $\min$--type close
formulae) is the novelty of the present manuscript.

\subsection{The testing functions for $d_{{\bf s}}$}
Obviously, the above theorem is valuable only if $\tau_{{\bf s}}$
is essentially  different than $d_{{\bf s}}$ and also if it is computable from
other different geometrical  behaviours. It is also clear that
not any  upper bound  $d_{{\bf s}}\leq \tau_{{\bf s}}$
satisfies the testing property {\it (2)}: this is satisfied only for bounds $\tau({\bf s})$ with
very structural relationship, symbiosis  with the original $d_{{\bf s}}$.
Hence it is not easy to find
testing functions, they must `testify'  about some deep geometric property:
even the  existence of
computable testing function(s) is really remarkable.

Our first test function is defined as follows. Consider again $Z\geq E$, $l'\in-\calS'$ associated with a resolution $\tX$, as above. Then, besides the Abel map $c^{l'}(Z)$ one can consider its `multiples' $\{c^{nl'}(Z)\}_{n\geq 1}$. It turns out that $n\mapsto \dim \im (c^{nl'}(Z))$
is a non-decreasing sequence,  $\im (c^{nl'}(Z))$ is an affine subspace
for $n\gg 1$, whose dimension $e_Z(l')$ is independent of $n\gg 0$, and essentially it depends only
on the $E^*$--support of $l'$ (i.e., on $I\subset \calv$, where $-l'=\sum_{v\in I}a_vE^*_v$ with all
$\{a_v\}_{v\in I}$ nonzero). From construction $d_Z(l')\leq e_Z(l')$, however they usually are not the same. Furthermore,
$e_Z(l')=e_Z(I)$ plays a crucial role in different analytic properties of $\tX$
(surgery formula, $h^1(\calL)$--computations, base point freeness properties). For details see \cite{NNI} or
subsections \ref{ss:AbelMap} and \ref{ss:LD} here, especially definition \ref{def:6.2} and Theorem
\ref{prop:AZ} (and also the proof of Theorem \ref{th:ALGORITHM}).
Now, at any step of the tower $\tX_{{\bf s}}$ one can consider this invariant
$e_{Z_{{\bf s}}}(l'_{{\bf s}})$, an integer denoted by $e_{{\bf s}}$.

Theorem \ref{th:ALGORITHM} (the `first algorithm')
guarantees that $e_{{\bf s}}$ is a testing function for $d_{{\bf s}}$.

The invariants $\{e_{{\bf s}}\}_{{\bf s}}$ are still hard to compute (cf. \ref{ss:4.1}).
However, the first algorithm is a necessary intermediate step for the second algorithm, valid  for
another testing function.

The advantage of the second testing function is that it is defined at the level of $\tX$ only.
It is based on  Laufer's  perfect pairing $H^1(\calO_Z)\otimes \calG_Z
\to \C$, where $\calG_Z$ denoted the  space of classes of forms
$H^0(\tX, \Omega^2_{\tX}(Z))/H^0(\tX,\Omega^2_{\tX})$. $\calG_Z$ has a natural divisorial filtration
$\{\calG_l\}_{0\leq l\leq Z}$, where $\calG_l$ is generated by forms with pole $\leq l$.
Its dimension (via Laufer duality) is $h^1(\calO_l)$.
(For more see \cite{NNI} and \ref{ss:LD} here.)
Next, for any ${\bf s}$  define the cycle $l_{{\bf s}}\in L$ of $\tX$ by
$$l_{{\bf s}}:=\min \Big\{
\sum_{v\in \calv} \, \min_{1\leq k_v\leq a_v}\{{\bf s}_{v, k_v}\}E_v, Z\Big\}\in L.$$
Set also $g_{{\bf s}}:=\dim \calG_{l_{\bf s}}$ as well. It turns out (see \ref{ss:4.1})
that
$d_{{\bf s}}\leq e_{{\bf s}}\leq h^1(\calO_Z)-g_{{\bf s}}$.
Usually, the equality $e_{{\bf s}}= h^1(\calO_Z)-g_{{\bf s}}$ rarely happens, however, it happens
whenever the testing property requires it! Theorem \ref{th:ALGORITHM4} (the `second algorithm')
says  that $h^1(\calO_Z)-g_{{\bf s}}$ is  a testing function for $d_{{\bf s}}$ indeed.

The cases of superisolated  singularities is exemplified.

The second algorithm has several consequences. E.g., a `numerical'  one, cf. (\ref{eq:form3}):
\begin{equation*}
d_Z(l') = \min_{0\leq Z_1 \leq Z}\{\, (l', Z_1) + h^1(\calO_Z) - h^1(\calO_{Z_1})\, \}, \ \mbox{or}, \
{\rm codim}\, \im (c^{l'}(Z))=\max_{0\leq Z_1 \leq Z}\{\, h^1(\calO_{Z_1})-(l', Z_1) \, \}.
\end{equation*}
The cycles $Z_1$ for which the above minimum is realized have several additional geometric
properties
(cf.  Lemma \ref{lem:MINSETS} and \ref{ss:5.1}). In particular, such a $Z_1$ imposes the following
 conceptual consequence:

\vspace{1mm}

\noindent {\bf Structure Theorem for the image of the Abel map.}
{\it Fix a resolution $\tX$, a cycle $Z\geq  E$  and a   Chern class $l'\in -\calS'$
as above.
Then there exists an effective cycle $Z_1 \leq Z$, such that: {\it (i)} the map
$\eca^{l'}(Z) \to H^1(Z_1)$ is birational onto its image,
and {\it (ii) }the generic fibres of the restriction of $r$,
$r^{im}:\im(c^{l'}(Z)) \to \im(c^{l'}(Z_1))$,  have dimension $h^1(\calO_{Z})-h^1(\calO_{Z_1})$.
In particular, for any such $Z_1$, the space $\im (c^{l'}(Z))$ is birationally
equivalent with an affine fibration  over $\eca^{l'}(Z_1)$ with affine fibers of dimension
$h^1(\calO_{Z})-h^1(\calO_{Z_1})$.}

\subsection{The case of generic analytic structure}  In section \ref{ss:GEN} we prove that if $\tX$ has a generic analytic structure (in the sense of \cite{LaDef1,NNII}), and $Z\geq E$ and $l'\in-\calS' $ then
both $\dim \im (c^{l'}(Z))$ and ${\rm codim} \im (c^{l'}(Z))$ are topological. E.g., we have
(where $\chi$ is the usual Riemann--Roch expression):
\begin{equation}\label{eq:gen3int}
{\rm codim}\, \im (c^{l'}(Z))=  \max_{0\leq Z_1\leq Z}\big\{ \, -(l', Z_1)
-  \chi(Z_1) +  \chi(E_{|Z_1|}) \, \big\}.
\end{equation}
The maximum at the right hand side is realized e.g. for the cohomology cycle of
$\calL^{im}_{gen}\in \im (c^{l'}(Z))\subset \pic^{l'}(Z)$.
Furthermore,
\begin{equation*}\label{eq:gen4int}
h^1(Z, \calL)\geq  \max_{0\leq Z_1\leq Z}\big\{ \, -(l', Z_1)
-  \chi(Z_1) +  \chi(E_{|Z_1|}) \, \big\}
\end{equation*}
for any $\calL\in \im (c^{l'}(Z))$ and equality holds for generic $\calL_{gen}^{im}\in
 \im (c^{l'}(Z))$.

The identity (\ref{eq:gen3int}),
valid for a generic analytic structure of $\tX$, extends to an optimal inequality
valid for {\it any analytic structure}.

\begin{theorem}\label{th:tzint}
Consider an {\em arbitrary}  normal surface singularity $(X,o)$, its resolution $\tX$, $Z\geq E$ and
$l'\in -\calS'$. Then
${\rm codim}\, \im (c^{l'}(Z))=h^1(Z, \calL^{im}_{gen})$
satisfies
\begin{equation}\label{eq:gen5int}
{\rm codim}\, \im (c^{l'}(Z))\geq  \max_{0\leq Z_1\leq Z}\big\{ \, -(l', Z_1)
-  \chi(Z_1) +  \chi(E_{|Z_1|}) \, \big\}.
\end{equation}
In particular, for any $\calL\in \im (c^{l'}(Z))$ one also has
\begin{equation*}\label{eq:gen5bint}
h^1(Z,\calL)\geq h^1(Z,\calL_{gen}^{im})={\rm codim}\, \im (c^{l'}(Z))\geq  \max_{0\leq Z_1\leq Z}\big\{ \, -(l', Z_1)
-  \chi(Z_1) +  \chi(E_{|Z_1|}) \, \big\}.
\end{equation*}
\end{theorem}
The right hand side of (\ref{eq:gen5int}) is a sharp  topological lower bound for
${\rm codim}\, \im (c^{l'}(Z))$. The inequality (\ref{eq:gen5int}) can also be interpreted as the semi-continuity statement
$${\rm codim}\, \im( c^{l'}(Z))(\mbox{arbitrary analytic structure})\geq
{\rm codim}\, \im( c^{l'}(Z))(\mbox{generic analytic structure}).$$

\subsection{Generalization.} Sections \ref{s:projAbel} and \ref{s:twisted} target
generalizations of the previous parts, valid for $\{h^1(Z,\calL)\}_{\calL\in \im c^{l'}(Z)}$,
to the shifted case, valid for  $\{h^1(Z,\calL_0\otimes\calL)\}_{\calL\in \im c^{l'}(Z)}$,
where $\calL_0 \in \pic^{l'_0}(Z)$ is a fixed bundle without fixed components. In order
to run a parallel theory based on Abel maps, we have to create the {\it new Abel map}
$c^{l'}_{\calL_0}(Z):\eca^{l'}(Z)\to \pic^{l'}_{\calL_0}(Z)$,
where $\pic^{l'}_{\calL_0}(Z)$ is an affine space associated with the vector space
$\pic^{0}_{\calL_0}(Z)\simeq H^1(Z,\calL_0)$. ($\pic^{l'}_{\calL_0}(Z)$ appears also
as an affine quotient of the classical  $\pic^{l'}(Z)$ as well.) Section \ref{s:projAbel} contains the definitions and the needed exact sequences. Section \ref{s:twisted} contains the extension of the
two algorithms to  this situation.

\section{Preliminaries}\label{s:prel}

\subsection{Notations regarding a good resolution}  \cite{Nfive,trieste,NCL,NNI}
Let $(X,o)$ be the germ of a complex analytic normal surface singularity,
 and let us fix  a good resolution  $\phi:\widetilde{X}\to X$ of $(X,o)$.
Let $E$ be  the exceptional curve $\phi^{-1}(0)$ and  $\cup_{v\in\calv}E_v$ be
its irreducible decomposition. Define  $E_I:=\sum_{v\in I}E_v$ for any subset $I\subset \calv$.

We will assume that each $E_v$ is rational, and the dual graph is a tree. This happens exactly when the link
$M$ of $(X,o)$ is a rational homology sphere.

$L:=H_2(\widetilde{X},\mathbb{Z})$, endowed
with a negative definite intersection form  $(\,,\,)$, is a lattice. It is
freely generated by the classes of  $\{E_v\}_{v\in\mathcal{V}}$.
 The dual lattice is $L'={\rm Hom}_\Z(L,\Z)=\{
l'\in L\otimes \Q\,:\, (l',L)\in\Z\}$. It  is generated
by the (anti)dual classes $\{E^*_v\}_{v\in\mathcal{V}}$ defined
by $(E^{*}_{v},E_{w})=-\delta_{vw}$ (where $\delta_{vw}$ stays for the  Kronecker symbol).
$L'$ is also  identified with $H^2(\tX,\Z)$, where the first Chern classes live.

All the $E_v$--coordinates of any $E^*_u$ are strict positive.
We define the Lipman cone as $\calS':=\{l'\in L'\,:\, (l', E_v)\leq 0 \ \mbox{for all $v$}\}$.
As a monoid it is generated over $\bZ_{\geq 0}$ by $\{E^*_v\}_v$.

$L$ embeds into $L'$ with
 $ L'/L\simeq H_1(M,\mathbb{Z})$,  abridged by $H$.
Each class $h\in H=L'/L$ has a unique representative $r_h\in L'$ in the semi-open cube
$\{\sum_vr_vE_v\in L'\,:\, r_v\in \bQ\cap [0,1)\}$, such that its class  $[r_h]$ is $h$.

There is a natural (partial) ordering of $L'$ and $L$: we write $l_1'\geq l_2'$ if
$l_1'-l_2'=\sum _v r_vE_v$ with all $r_v\geq 0$. We set $L_{\geq 0}=\{l\in L\,:\, l\geq 0\}$ and
$L_{>0}=L_{\geq 0}\setminus \{0\}$.

The support of a cycle $l=\sum n_vE_v$ is defined as  $|l|=\cup_{n_v\not=0}E_v$.

The {\it (anti)canonical cycle} $Z_K\in L'$ is defined by the
{\it adjunction formulae}
$(Z_K, E_v)=(E_v,E_v)+2$ for all $v\in\mathcal{V}$.
We write $\chi:L'\to \Q$ for the (Riemann--Roch) expression $\chi(l'):= -(l', l'-Z_K)/2$.

\bekezdes \label{bek:natline} {\bf Natural line bundles.} Let
 $\phi:(\tX,E)\to (X,o)$ be as above.
Consider the `exponential'  cohomology exact sequence (with $H^1(\tX, \calO_{\tX}^*)=\pic(\tX)$, the
group of isomorphic classes of holomorphic line bundles on $\tX$,  and
$H^1(\tX, \calO_{\tX})=\pic^0(\tX)$)
\begin{equation}\label{eq:exp}
0\to {\rm Pic}^0(\tX)\longrightarrow
{\rm Pic}(\tX)\stackrel{c_1}{\longrightarrow}H^2(\tX,\Z)\to 0.
\end{equation}
Here $c_1(\calL)\in H^2(\tX,\Z)=L' $ is the first Chern class of $\calL\in \pic(\tX)$.
Since $H^1(M,\Q)=0$, $\pic^0(\tX)\simeq H^1(\tX,\calO_{\tX})\simeq \C^{p_g}$,
where $p_g$ is the geometric genus. Write also  $\pic^{l'}(\tX)=c_1^{-1}(l')$.
Furthermore, see e.g. \cite{OkumaRat,trieste},
there exists a unique homomorphism (split)
 $s_1:L'\to {\rm Pic}(\tX)$  of $c_1$, that is  $c_1\circ s_1=id$, such that
$s_1$ restricted to $L$ is $l\mapsto \calO_{\tX}(l)$.
The line bundles $s_1(l')$ are called {\it natural line
bundles } of $\tX$. For several definitions of
them see  \cite{trieste}.
E.g., $\calL$ is natural if and only if one of its power has the form $\calO_{\tX}(l)$
for some {\it integral} cycle $l\in L$ supported on $E$.
In order to have a uniform notation we write $\calO_{\tX}(l')$ for $s_1(l')$ for any $l'\in L'$.

 For any $Z\geq E$ let $\calO_Z(l')$ be the
restriction of the natural line bundle $\calO_{\tX}(l')$ to $Z$. In fact, $\calO_Z(l')$
can be defined in an identical way as $\calO_{\tX}(l')$ starting from
the exponential cohomological sequence
$0\to \pic^0(Z)\to  \pic(Z)\to H^2(\tX,\Z)\to 0$ as well. Set also $\pic^{l'}(Z)=c_{1,Z}^{-1}(l')$.

\subsection{The Abel map \cite{NNI}}\label{ss:AbelMap}
For any $Z\geq E$ let
$\eca(Z)$  be the space of (analytic) effective Cartier divisors on 
$Z$. Their supports are zero--dimensional in $E$.
Taking the line bundle  of a Cartier divisor provides  the {\it Abel map}
$c=c(Z):\eca(Z)\to \pic(Z)$.
Let
$\eca^{l'}(Z)$ be the set of effective Cartier divisors with
Chern class $l' \in L'$, i.e.
$\eca^{l' }(Z):=c^{-1}(\pic^{l'}(Z))$.
The restriction of $c$ is denoted by   $c^{l'}:\eca^{l'}(Z)\to \pic^{l'}(Z)$.


A line bundle $\calL\in \pic^{l'}(Z)$ is in the image
 ${\rm im}(c^{l'})$ if and only if it has a section without fixed components, that is,
 if $H^0(Z,\calL)_{reg}\not=\emptyset $, where
$H^0(Z,\calL)_{reg}:=H^0(Z,\calL)\setminus \cup_v H^0(Z-E_v, \calL(-E_v))$.
By this definition (see (3.1.5) of \cite{NNI}) $\eca^{l'}(Z)\not=\emptyset$ if and only if
$-l'\in \calS'\setminus \{0\}$. It is advantageous to have a similar statement for
$l'=0$ too, hence we redefine  $\eca^0(Z)$ as $\{\emptyset\}$, a set/space with one element
(the empty divisor), and $c^0:\eca ^0(Z)\to \pic^0(Z)$ by $c^0(\emptyset)=\calO_Z$.
In particular,
\begin{equation}\label{eq:Chernzero1}
H^0(Z,\calL)_{reg}\not=\emptyset\ \Leftrightarrow\ \calL=\calO_Z\
\Leftrightarrow\ \calL\in {\rm im}(c^0) \ \ \mbox{ whenever $c_1(\calL)=0$}.\end{equation}
Hence, the extended statement valid for any $l'$ is:
\begin{equation}\label{eq:Chernzero}
 \eca^{l'}(Z)\not=\emptyset\ \Leftrightarrow \ -l'\in \calS'.
\end{equation}
Sometimes  even for $\calL\in\pic^{l'}(\tX)$ we write
$\calL\in \im(c^{l'})$ whenever $\calL|_Z\in  \im(c^{l'}(Z))$ for some $Z\gg0$. This happens
if and only if
$\calL\in \pic(\tX) $ has no fixed components.

It turns out that
$\eca^{l'}(Z)$ ($-l'\in\calS'$) is a smooth complex algebraic variety  of dimension $(l',Z)$ and
the Abel map  is an algebraic regular map. For more properties
and applications see \cite{NNI,NNII}.

\bekezdes\label{bek:modAbel} {\bf The modified Abel map.}  Multiplication by $\calO_{Z}(-l')$  gives
an  isomorphism of the affine spaces $\pic^{l'}(Z)\to \pic^0(Z)$. Furthermore, we identify
(via the exponential exact sequence) $\pic^0(Z)$ with the vector space
$H^1(Z, \calO_Z)$.

It is convenient to replace the Abel map $c^{l'}$ with the composition
$$\widetilde{c}^{l'}:\eca^{l'}(Z)\stackrel{c^{l'}}{\longrightarrow} \pic^{l'}(Z)\stackrel{\calO_Z(-l')}
{\longrightarrow} \pic^0(Z)\stackrel{\simeq}{\longrightarrow} H^1(\calO_Z).$$
The advantage of this new set of maps is that all the images sit in the same
vector  space $H^1(\calO_Z)$.

Consider the natural additive structure
$s^{l'_1,l'_2}(Z):\eca^{l'_1}(Z)\times \eca^{l'_2}(Z)\to \eca^{l'_1+l'_2}(Z)$ ($l'_1,l'_2\in-\calS'$)
provided by the sum of the divisors. One verifies (see e.g. \cite[Lemma 6.1.1]{NNI}) that
$s^{l'_1,l'_2}(Z)$ is dominant and quasi--finite. 
There is a parallel multiplication $\pic^{l'_1}(Z)\times \pic^{l'_2}(Z)\to \pic^{l'_1+l'_2}(Z)$,
$(\calL_1,\calL_2)\mapsto \calL_1\otimes\calL_2$, which satisfies
 $c^{l_1'+l'_2}\circ s^{l'_1,l'_2}= c^{l'_1} \otimes c^{l'_2}$ in $\pic^{l'_1+l'_2}$.
This, in the modified case, using
$\calO_Z(l'_1+l'_2)=\calO_Z(l'_1)\otimes \calO_Z(l'_2)$,
 reads as
 $\widetilde{c}^{l_1'+l'_2}\circ s^{l'_1,l'_2}= \widetilde{c}^{l'_1}
+ \widetilde{c}^{l'_2}$ in $H^1(\calO_Z)$.

\begin{definition}\label{def:VZ}
For any  $l'\in-\calS'$  let $A_Z(l')$
be the smallest dimensional affine subspace of
$H^1(\calO_Z)$ which contains $\im (\widetilde{c}^{l'})$. Let $V_Z(l')$,
 be the parallel vector subspace
of $H^1(\calO_Z)$,
the translation of  $A_Z(l')$ to the origin. 

For any $I\subset \calv$, $I\not=\emptyset$,
let $(X_I,o_I)$ be the multigerm $\tX/_{\cup_{v \in I}E_v}$
at its  singular points,
obtained by contracting the connected components of
$\cup_{v \in I}E_v$ in $\tX$. If $I=\emptyset$ then by convention $(X_I,o_I)$ is a smooth germ.
\end{definition}

\begin{theorem}\label{prop:AZ} \ \cite[Prop. 5.6.1, Lemma 6.1.6 and Th. 6.1.9]{NNI}
Assume that  $Z\geq E$.

(a)   For any $-l'=\sum_va_vE^*_v\in \calS'$ let the $E^*$--support of $l'$ be
$I(l'):=\{v\,:\, a_v\not=0\}$. Then  $V_Z(l')$ depends only on $I(l')$.
(This motivates to write $V_Z(l')$ as $V_Z(I)$ where $I=I(l')$.)

(b)  $V_Z(I_1\cup I_2)=V_Z(I_1)+V_Z(I_2)$ and $A_Z(l_1'+l_2')= A_Z(l_1')+A_Z(l_2')$.

(c)  $\dim V_Z(I)=h^1(\calO_Z)-h^1(\calO_{Z|_{\calv\setminus I}})$.

(d)  If $\calL^{im}_{gen}$ is a generic bundle of \,$\im (c^{l'})$
then $h^1(Z,\calL^{im}_{gen})=h^1(\calO_Z)-\dim(\im (c^{l'}))$.

(e)  For $n\gg 1$ one has $\im(\widetilde{c}^{nl'})=A_Z(nl')$, and $h^1(Z,\calL)=
h^1(\calO_Z)-\dim V_Z(I)=
 h^1(\calO_{Z|_{\calv\setminus I}})$
for any $\calL\in \im (c^{nl'})$.

\end{theorem}

 For different geometric reinterpretations of $\dim V_Z(I)$ see  also \cite[\S 9]{NNI}.

\subsection{} \label{ss:DOMINANT}
Theorem 4.1.1 of \cite{NNI} says that $c^{l'}(Z)$ is dominant if and only if
$ \chi(-l')<\chi(-l'+l)$ for any $0<l\leq Z$. In particular, the  dominance of
$c^{l'}(Z)$ is a topological property. If $c^{l'}(Z)$ is dominant then  $c^{l'}(Z')$
is dominant for any $0< Z'\leq Z$.

\subsection{Review of Laufer Duality \cite{Laufer72},
\cite[p. 1281]{Laufer77}}\label{ss:LD} \
Following Laufer, we identify the dual space $H^1(\tX,\cO_{\tX})^*$ with the space of global holomorphic
2-forms on $\tX\setminus E$ up to the subspace of those forms which can be extended
holomorphically over $\tX$.

For this, use first  Serre duality
$H^1(\tX,\cO_{\tX})^*\simeq
H^1_c(\tX,\Omega^2_{\tX})$. Then,   in the exact sequence
$$0\to H^0_c(\tX,\Omega^2_{\tX}) \to
 H^0(\tX,\Omega^2_{\tX}) \to
H^0(\tX\setminus E,\Omega^2_{\tX})\to
H^1_c(\tX,\Omega^2_{\tX})\to
H^1(\tX,\Omega^2_{\tX})$$
$H^0_c(\tX,\Omega^2_{\tX})=H^2(\tX,\cO_{\tX})^*=0$ by dimension argument, while $H^1(\tX,\Omega^2_{\tX})=0$ by the Grauert--Riemenschneider
vanishing. Hence,
\begin{equation}\label{eq:LD}
H^1(\tX,\cO_{\tX})^*\simeq  H^1_c(\tX,\Omega^2_{\tX})\simeq
H^0(\tX\setminus E,\Omega^2_{\tX})/ H^0(\tX,\Omega^2_{\tX}).\end{equation}

\bekezdes \label{bek:Z}
Above $H^0(\tX\setminus E,\Omega^2_{\tX})$ can be replaced by
$H^0(\tX,\Omega^2_{\tX}(Z))$ for a large cycle $Z$
(e.g. for $Z\geq \lfloor Z_K \rfloor$). Indeed, for any cycle $Z>0$ from the exacts sequence of sheaves
$0\to\Omega^2_{\tX}\to \Omega^2_{\tX}(Z)\to \calO_{Z}(Z+K_{\tX})\to 0$ and from the vanishing
$h^1(\Omega^2_{\tX})=0$ and Serre duality one has
\begin{equation}\label{eq:duality}
H^0(\Omega^2_{\tX}(Z))/H^0(\Omega^2_{\tX})=H^0(\calO_Z(Z+K_{\tX}))\simeq H^1(\calO_Z)^*.
\end{equation}
Since $H^1(\calO_Z)\simeq H^1(\calO_{\tX})$ for $Z\geq \lfloor Z_K \rfloor$, the natural inclusion
\begin{equation}\label{eq:inclusion}
H^0(\Omega^2_{\tX}(Z))/H^0( \Omega^2_{\tX})\hookrightarrow
H^0(\tX\setminus E, \Omega^2_{\tX})/H^0(\Omega^2_{\tX})
\end{equation}
is an isomorphism.

This pairing reduces  to a perfect pairing at the level of an arbitrary $Z>0$, cf. \cite[7.4]{NNI}.
 Indeed, consider the above perfect pairing  $\langle \cdot,\cdot \rangle:
H^1(\tX,\calO_{\tX})\otimes H^0(\tX\setminus E, \Omega^2_{\tX})/H^0(\Omega^2_{\tX})\to \C$
given via integration
 of class representatives.  In $H^1(\tX,\calO_{\tX})$ let $A$ be the image of $H^1(\tX,\calO_{\tX}(-Z))$, hence $H^1(\tX,\calO_{\tX})/A=H^1(\calO_Z)$. On the other hand,
 in $H^0(\tX\setminus E, \Omega^2_{\tX})/H^0(\Omega^2_{\tX})$ consider the subspace $B:=H^0(\Omega^2_{\tX}(Z))/H^0(\Omega^2_{\tX})$ of dimension $h^1(\calO_Z)$
(cf. (\ref{eq:duality})). Since $\langle A,B\rangle=0$, the pairing factorizes to a perfect
pairing $H^1(\calO_Z)\otimes H^0(\Omega^2_{\tX}(Z))/H^0(\Omega^2_{\tX})\to \C$.
It can be described by the very same integral form of the corresponding class representatives.

\bekezdes \label{bek:diff forms} {\bf The linear subspace arrangement $\{V_Z(I)\}_I\subset
H^1(\calO_Z)$ and differential forms.}
 The arrangement $\{V_Z(I)\}_I$ transforms into a
  linear subspace arrangement of $H^0(\Omega^2_{\tX}(Z))/H^0(\Omega^2_{\tX})$ via the (Laufer)
  non--degenerate pairing $H^1(\calO_Z)\otimes   H^0(\Omega^2_{\tX}(Z))/H^0(\Omega^2_{\tX})\to\C$
 as follows. Let $\Omega_Z(I)$ be the subspace
  $H^0(\Omega^2_{\tX}(Z|_{\calv\setminus I}))/H^0(\Omega^2_{\tX})$ in
  $H^0(\Omega^2_{\tX}(Z))/H^0(\Omega^2_{\tX})$, that is, the subspace generated by those forms which
  have no poles along generic points of any $E_v$, $v\in I$.
  \begin{proposition}\label{prop:HDIZ} \ \cite[8.3]{NNI} Via Laufer duality
$V_Z(I)= \Omega_Z(I)^\perp=\{x:\langle x,\Omega_Z(I)\rangle=0\}$ for $Z\geq E$.
\end{proposition}

\bekezdes\label{bek:OmegaD}
  Furthermore, for any $l'\in -\calS'\setminus \{0\}$ consider a divisor
  $D \in \eca^{l'}(Z) $, which is a union of $(l', E)$ disjoint divisors  $\{D_i\}_i $,
 each of them $\cO_Z$--reduction of reduced divisors $\{\widetilde{D}_i\}_i $ of $\tX$
 intersecting  $E$  transversally.
Set  $\widetilde{D}=\cup_i\widetilde{D}_i $ and  $\calL:=\widetilde{c}^{l'}(D)\in  H^1(\calO_Z)$.
Write  also $Z=\sum_{v\in\calv}r_vE_v$.

We introduce a subsheaf $\Omega_{\tX}^2(Z)^{{\rm regRes}_{\widetilde{D}}}$
of  $\Omega_{\tX}^2(Z)$ consisting of those forms $\omega$
which have the property that the residue ${\rm Res}_{\widetilde{D}_i}(\omega)$
has no poles along $\widetilde{D}_i$ for all $i$. This means that the restrictions of  $\Omega_{\tX}^2(Z)^{{\rm regRes}_{\widetilde{D}}}$ and    $\Omega_{\tX}^2(Z)$
on the complement of the support of $\widetilde{D}$ coincide, however along $\widetilde{D}$
one has the following local picture.
Introduce near $p=E\cap \widetilde{D}_i=E_{v_i}\cap \widetilde{D}_i$   local coordinates
$(u,v)$ such that  $\{u=0\}=E$ and  $\widetilde{D}_i$ has  local equation $v$. Then a local section of $\Omega_{\tX}^2(Z)$  in this system
has the form $\omega=\sum_{ k \geq -r_{v_i}, j\geq 0} a_{k,j}u^k v^j du\wedge dv$. Then, by definition,
 the residue ${\rm Res}_{\widetilde{D}_i}(\omega)$ is
$(\omega/dv)|_{v=0}=\sum_k a_{k,0}u^k du$, hence the pole--vanishing reads as $a_{k,0}=0$ for all $k<0$.
Note that $\Omega_{\tX}^2(Z-\widetilde{D}) $ and the sheaf of regular forms  $\Omega_{\tX}^2$ are
subsheaves of $\Omega_{\tX}^2(Z)^{{\rm regRes}_{\widetilde{D}}}$.

Set $\Omega _Z(D):=H^0(\tX,\Omega_{\tX}^2(Z)^{{\rm regRes}_{\widetilde{D}}})/ H^0(\tX,\Omega_{\tX}^2)$.
This can be regarded as a subspace of $H^1(\calO_Z)^*=H^0(\tX,\Omega^2_{\tX}(Z))/H^0(\tX,\Omega_{\tX}^2)
$.

\begin{theorem}\label{th:Formsres} \ \cite[Th. 10.1.1]{NNI}
In the above situation one has the following facts.

(a) The sheaves $\Omega_{\tX}^2(Z)^{{\rm regRes}_{\widetilde{D}}}/\Omega_{\tX}^2$ and $\calO_{Z}(K_{\tX}+Z-D)$ are isomorphic.

(b) $H^1(Z,\calL)^*\simeq\Omega_Z(D)$.

(c) The image $(T_D\widetilde{c}) (T_{D} \eca^{l'}(Z))$ of the tangent map at $D$ of
$\widetilde{c}:\eca^{l'}(Z)\to H^1(\calO_Z)$ is the intersection of kernels of
linear maps $T_{\calL}\omega:T_{\calL}H^1(\calO_Z)\to\bC$, where $\omega\in
H^0(\tX,\Omega_{\tX}^2(Z)^{{\rm regRes}_{\widetilde{D}}})$.
\end{theorem}

If $I$ is the $E^*$--support of $l'$ (that is, $\widetilde{D}$ intersects $E$ exactly along
$\cup_{v\in I}E_v$), then $\Omega_Z(I)\subset \Omega_Z(D)\subset H^1(\calO_Z)^*$.
Dually, via Proposition \ref{prop:HDIZ} and
Theorem \ref{th:Formsres}{\it (c)} (and up to a linear translation of $\im (T_D\widetilde{c})$)
\begin{equation}\label{eq:dual}
(T_D\widetilde{c})(T_D\eca^{l'}(Z))=\Omega _Z(D)^\perp\subset \Omega_Z(I)^\perp=V_Z(I)\subset H^1(\calO_Z).
\end{equation}

Let us fix a point $p\in E$ and a local coordinate system $(u,v)$ around $p$ such that
$E=\{u=0\}$, cf. \ref{bek:OmegaD}. Fix also some  $\omega\in H^0(\tX,\Omega_{\tX}^2(Z))$ which has pole of order $o>0$ at the exceptional divisor in $E$ containing $p$.
We say that (the divisor of) $\omega$ has no support point at $p$ if it can be represented locally as
$(\varphi(u,v)/u^o)du\wedge dv$ with $\varphi$ holomorphic and $\varphi(0,0)\not=0$. The other points are the support points denoted by ${\rm supp}(\omega)$.

\begin{lemma}\label{lem:dualInt}
Fix $\omega\in H^0(\tX,\Omega_{\tX}^2(Z))$ such that there exists  a point $p\in E_v$,
a local
divisor $\widetilde{D}_1$
in $\tX$  with the following properties:
 (a) $\widetilde{D}_1$ is  part of certain
$\widetilde{D}=\widetilde{D}_1+ \widetilde{D}_2$, such that
$\widetilde{D}_1\cap E=\widetilde{D}_1\cap E_v =p\not\in \widetilde{D}_2\cup {\rm supp}(\omega)$, and (b) $\widetilde{D}$ is  a lift of $D\in\eca^{l'}(Z)$, and the
class of $\omega$  in
 $H^0(\tX,\Omega_{\tX}^2(Z))/H^0(\tX,\Omega_{\tX}^2)$
restricted on
 $\im T_{D}\widetilde{c}^{l'}(Z)$ is zero. Then $\omega$ has no pole along $E_v$.
\end{lemma}
\begin{proof} Assume that $\omega $ has a pole of order $o>0$ along $E_v$. Fix some local
coordinated $(u,v)$ at $p:=\widetilde{D}_1\cap E_v$  such that $\omega $ locally is
$du\wedge dv/u^o$ and $\widetilde{D}_1$ is $\{g(u,v)=0\}$.
A deformation $g_t(u,v)$ of $g$ produces
a tangent vector in $T_D\eca^{l'}(Z)$
and the action of $\omega$ on it is given by
(for details see \cite[7.2]{NNI})
\begin{equation}\label{eq:int}
\frac{d}{dt}\Big|_{t=0}\
\int_{|u|=\epsilon, \, |v|=\epsilon} \
\log \frac{g_t(u,v)}{g(u,v) } \cdot \frac{du\wedge dv}{u^o}.
\end{equation}
Hence if we realize a deformation $g_t$ for which the expression from (\ref{eq:int})
is non--zero, we get a contradiction.
Note that $g$ necessarily has the form $cv^k + \sum_{n>k}c_nv^n+uh(u,v)=cv^k+h'$ for some $k\geq 1$, $c_n\in\C$  and $c\in\C^*$. Then set $g_t=c(v-tu^{o-1})^k +h'$. Then the $t$--coefficient of the integrant is
$\frac{kdu\wedge dv}{uv}\cdot (1-\frac{h'}{cv^k}+( \frac{h'}{cv^k})^2-\cdots )$, hence
(\ref{eq:int})  is non--zero.
\end{proof}

\begin{definition}\label{bek:filtrforms}
 Additionally to the linear subspace arrangement $\{\Omega_Z(I)\}_I\subset
H^0(\Omega^2_{\tX}(Z))/H^0(\Omega^2_{\tX})\simeq H^1(\calO_Z)^*$
we consider a more subtle object,  a filtration indexed by $l\in L$, $0\leq l\leq Z$ as well, called the {\it  multivariable
divisorial filtration of forms}. Indeed, for any such $l$ we define
$\calG_l:= H^0(\Omega^2_{\tX}(l))/H^0(\Omega^2_{\tX})\subset
H^0(\Omega^2_{\tX}(Z))/H^0(\Omega^2_{\tX})$, equivalent to
$H^1(\calO_l)^*\hookrightarrow H^1(\calO_Z)^*$, dual to the natural epimorphisms
$H^1(\calO_Z)\twoheadrightarrow H^1(\calO_l)$.
In particular, $\calG_l\simeq H^1(\calO_l)^*$.
$\calG_l$ is generated by forms with pole $\leq l$.
In particular, $\calG_0=0$, $\calG_Z$ is the total vector space,
$\calG_{l_1}\subset \calG_{l_2}$ whenever $l_1\leq l_2$, and
$\calG_{l_1}\cap \calG_{l_2}=\calG_{\min\{l_1,l_2\}}$.

Note that if $l=\sum_{v\not\in I} r_vE_v$ and all $r_v\gg 0$ then $\calG_{\min(l, Z)}=\Omega_Z(I)$.
\end{definition}

\section{The first algorithm for the computation of $\dim {\rm Im} (c^{l'}(Z))$}\label{s:GenAlg}

\subsection{} We fix $Z\geq E$ and $l'\in -\calS' $ as above.

\begin{definition}\label{def:6.2} For any $l'\in -\calS'$ with $E^*$--support $I$
($\emptyset \subset I\subset \calv$)  we set the following notations:
$e_Z(l')=e_Z(I):=\dim V_Z(l')=\dim V_Z(I)$ and $d_Z(l'):=\dim \im (c^{l'}(Z))$.
\end{definition}
 From definitions and Propositions \ref{prop:AZ} and \ref{prop:HDIZ} (see also (\ref{eq:dual}))
 \begin{equation}\label{eq:6.2b}
\begin{split}
& d_Z(l')\leq e_Z(l')\\
& e_Z(I)=h^1(\calO_Z)-h^1(\calO_{Z|_{\calv\setminus I}})=h^1(\calO_Z)-\dim \Omega_Z(I).
 \end{split}\end{equation}

Usually $d_Z(l')\not=e_Z(l')$. Next statement provides a
 criterion for the validity of the equality.

\begin{lemma}\label{lem:imcA}
Let  $l' \in -S'$ with $E^*$--support $I$ and $Z \geq E$.
Assume that  $\calL$ is  a regular value  of $\widetilde{c}^{l'}$  in
 $\im( \widetilde{c}^{l'})$ such that
 for any $\omega\in H^0(\tX,\Omega_{\tX}^2(Z))$ there exists a section
 $s\in H^0(\calL)_{reg}$ such that ${\rm div}(s)\cap {\rm supp}(\omega)=\emptyset$.
(This is guaranteed e.g. if
the bundle $\calL$ has no base points.)  Then
 $T_{\calL}(\im \widetilde{c}^{l'})=A_Z(l')$, hence  $d_Z(l')= e_Z(l')$.
\end{lemma}
\begin{proof}
Since $\calL$ is a regular value, $\calL$ is a smooth point of $\im (\widetilde{c}^{l'})$ and
$T_{\calL}
\im (\widetilde{c}^{l'}) =\im (T_D \widetilde{c}^{l'})$
for {\it any}  $D\in (\widetilde{c}^{l'})^{-1}(\calL)$ (cf. \cite[3.3.2]{NNI}).
We have to prove that $T_{\calL} \im (\widetilde{c}^{l'}) = A_Z(l')$;
 we prove the dual identity in the space of  forms, namely,
 $(T_{\calL} \im (\widetilde{c}^{l'})^\perp =\Omega_Z(I)$ (see  (\ref{eq:dual})).

Assume the contrary, that is, $( T_{\calL} \im (\widetilde{c}^{l'}))^\perp  \neq \Omega_Z(I)$. Since $\Omega_Z(I)\subset ( T_{\calL}
\im (\widetilde{c}^{l'}))^\perp $ (the
duality integral on $\Omega_Z(I)\times T_{\calL}\im (\widetilde{c}^{l'})$ is zero,
 cf. \cite[7.2]{NNI} or (\ref{eq:dual})) we get, that there is a form $\omega \in (T_{\calL}
\im (\widetilde{c}^{l'}))^\perp \setminus \Omega_Z(I)$.

Next choose  $D \in (\widetilde{c}^{l'})^{-1}(\calL)$ such that its lift $\widetilde{D}$ satisfies $\widetilde{D}\cap {\rm supp}(\omega)=\emptyset$. But  $\omega \in
(T_{\calL} \im (\widetilde{c}^{l'}))^\perp=
(\im (T_D \widetilde{c}^{l'}))^\perp $ and $\omega\not
\in \Omega_Z(I)$ contradict Lemma \ref{lem:dualInt}.
\end{proof}

%

In this section we provide an algorithm,  valid for any analytic structure, which determines
 $d_Z(l')$ in terms of a finite collection of invariants of type $e_Z(l')$,
 associated with a finite sequence of resolutions obtained via certain extra blowing ups
from $\tX$.

\subsection{Preparation for the algorithm}\label{bek:algsetup}
Fix some resolution $\tX$ of $(X,o)$ and
 $-l' = \sum_{v \in \calv} a_v E_v^*\in\calS'\setminus \{0\}$ (hence each $a_v\in\Z_{\geq 0}$).
 In the next construction  we will consider a finite sequence  of blowing ups starting from $\tX$.
In order to find a bound for the number of blowing ups recall that for any representative $\omega$ in
$H^0(\tX\setminus E,\Omega^2_{\tX})/H^0(\tX,\Omega^2_{\tX})$ the order of pole of $\omega$
along some $E_v$ is less than or equal to
 the $E_v$--multiplicity $m_v$ of $\max\{0,\lfloor Z_K\rfloor\} $ (see e.g.
\cite[7.1.3]{NNI} or \ref{ss:LD} here).
Then, for every $v\in \calv$ with $a_v>0$ we fix $a_v$ generic points on $E_v$, say
$p_{v,k_v}$, $1\leq k_v\leq a_v$. Starting from each $p_{v,k_v}$ we consider a sequence of blowing ups  of length  $m_v$:
first we blow up $p_{v,k_v}$ and we create the exceptional curve $F_{v,k_v,1}$,
then we blow up a generic point of
$F_{v,k_v,1}$ and we create $F_{v,k_v,2}$, and we do this all together $m_v$ times.
We proceed in this way with all points $p_{v,k_v}$, hence we get $\sum_va_v$ chains of modifications.
If $a_vm_v=0$ we do no modification along $E_v$.
A set of integers
${\bf s}=\{{\bf s}_{v,k_v}\}_{v\in\calv,\ 1\leq k_v\leq a_v}$ with $0\leq {\bf s}_{v,k}\leq m_v$ provides an
intermediate step of the tower: in the $(v,k_v)$ tower we do exactly ${\bf s}_{v,k_v}$ blowing ups;
${\bf s}_{v,k_v}=0$ means that we do not blow up $p_{v,k_v}$ at all.
(In the sequel, in order to avoid aggregation of indices, we simplify $k_v$ into $k$.)
Let us denote this modification by $\pi_{{\bf s}}:\tX_{{\bf s}}\to \tX$.
In $\tX_{{\bf s}}$ we find the exceptional curves $\cup_{v\in\calv}E_v\cup \cup_{v,k}
\cup_{1\leq t \leq {\bf s}_{v,k}}F_{v,k,t}$; we index the set of vertices as
$\calv_{{\bf s}}:=\calv \cup \cup_{v,k}
\cup_{1\leq t\leq {\bf s}_{v,k}}\{w_{v,k,t}\}$.
At each level ${\bf s}$ we set the next objects: $Z_{{\bf s}}:=\pi_{{\bf s}}^*(Z)$,
$I_{{\bf s}}:=\cup_{v,k}\{w_{v,k,{\bf s}_{v,k}}\}$, $-l'_{{\bf s}}:= \sum_{v,k}F^*_{v,k,{\bf s}_{v,k}}$
(in $L'_{{\bf s}}$, where
 $F_{v,k,0}=E_v$), $d_{{\bf s}}:=\dim \im c^{l'_{{\bf s}}}(Z_{{\bf s}})$
and  $e_{{\bf s}}:=e_{Z_{{\bf z}}}(I_{{\bf s}})$
(both considered in  $\tX_{{\bf s}}$).

By similar argument as in (\ref{eq:6.2b}) one has
 again $d_{{\bf s}}\leq e_{{\bf s}}$ for any ${\bf s}$.

From definitions, for ${\bf s}={\bf 0}$ one has
$I_{{\bf 0}}= |l'|$, $e_{{\bf 0}}=e_Z(l')$ and $d_{{\bf 0}}=d_Z(l')$.

There is a natural partial ordering on the set of ${\bf s}$--tuples. Some of the above invariants are
constant with respect to ${\bf s}$, some of them are only monotonous. E.g., by Leray spectral sequence
one has $h^1(\calO_{Z_{{\bf s}}})=h^1(\calO_Z)$ for all ${\bf s}$. One the other hand,
\begin{equation}\label{eq:ineqes}\mbox{
if ${\bf s}_1\leq {\bf s}_2$ then $e_{{\bf s}_1}=h^1(\calO_{Z_{{\bf s}_1}})-\dim
\Omega_{Z_{{\bf s}_1}}(I_{{\bf s}_1})\geq  h^1(\calO_{Z_{{\bf s}_2}})-\dim
\Omega_{Z_{{\bf s}_2}}(I_{{\bf s}_2})=  e_{{\bf s}_2}$}\end{equation}
because
$\Omega_{Z_{{\bf s}_1}}(I_{{\bf s}_1})\subset \Omega_{Z_{{\bf s}_2}}(I_{{\bf s}_2})$.
In fact, for any
$\omega$, the pole--order along $F_{v,k,{\bf s}_{v, k}+1}$ of its pullback is one less than the pole--order of
$\omega$ along  $F_{v,k,{\bf s}_{v, k}}$. Hence, for ${\bf s}={\bf m}$
(that is, when ${\bf s}_{v,k}=m_v$ for all $v$ and $k$, hence all the possible pole--orders along
$I_{{\bf m}}$ automatically vanish) one has
$\Omega_{Z_{{\bf m}}}(I_{{\bf m}})=H^0(\tX_{{\bf m}}, \Omega^2_{{\tX_{{\bf m}}}}(Z_{{\bf m}}))/
H^0(\Omega^2_{{\tX_{{\bf m}}}})$. Hence $e_{{\bf m}}=0$. In particular,  necessarily
$d_{{\bf m}}=0$ too.

More generally, for any ${\bf s}$ and $(v,k)$ let ${\bf s}^{v,k}$ denote that  tuple which
is obtained from
${\bf s}$ by increasing ${\bf s}_{v,k}$ by one. By the above discussion
if   no form has pole along $F_{v,k,{\bf s}}$ then
$\Omega_{Z_{{\bf s}}}(I_{{\bf s}})=\Omega_{Z_{{\bf s}^{v,k}}}(I_{{\bf s}^{v,k}})$, hence
$e_{{\bf s}}=  e_{{\bf s}^{v,k}}$. Furthermore, by Laufer duality (or,
integral presentation of the Abel map as in \cite[\S 7]{NNI}), under such condition
$d_{{\bf s}}=  d_{{\bf s}^{v,k}}$ as well.

Therefore, we can redefine $e_{{\bf s}}$ and  $d_{{\bf s}}$  for tuples
${\bf s}=\{{\bf s}_{v,k}\}_{v,k}$ even for
arbitrary ${\bf s}_{v,k}\geq 0$: $e_{{\bf s}} = e_{\min\{{\bf s},{\bf m}\}}$ and
$d_{{\bf s}} = d_{\min\{{\bf s},{\bf m}\}}$ (and these values agree with the ones which might be
obtained by the
 first original
construction  applied for larger chains of blow ups).

The next theorem relates the invariants  $\{d_{{\bf s}}\}_{{\bf s}}$ and
$\{e_{{\bf s}}\}_{{\bf s}}$.

\begin{theorem}\label{th:ALGORITHM} {\bf (First algorithm)}
 With the above notations  the following facts hold.

(1) $ d_{{\bf s}} -  d_{{\bf s}^{v, k}} \in \{0, 1\}$.

(2) If for some fixed ${\bf s}$ the numbers $\{d_{{\bf s}^{v, k}}\}_{v,k}$ are not the same,
then $d_{{\bf s}} = \max_{v, k}\{\,d_{{\bf s}^{v, k}}\}$.
 In the case when all the numbers $\{d_{{\bf s}^{v, k}}\}_{v,k}$ are the same,
 then if this common value $d_{{\bf s}^{v, k}}$ equals $e_{{\bf s}}$, then $d_{{\bf s}} =
e_{{\bf s}}  =d_{{\bf s}^{v, k}}$;  otherwise $d_{{\bf s}} = d_{{\bf s}^{v, k}}+1$.
\end{theorem}

The proof of Theorem \ref{th:ALGORITHM} together with the proof of Theorem \ref{th:ALGORITHM4}
(the `Second algorithm')
 from the next section will be given in a more general context in section \ref{s:twisted}.

\bekezdes\label{bek:abde}
Theorem \ref{th:ALGORITHM} is suitable to run a decreasing induction  over the entries of ${\bf s}$
in order to determine
$\{d_{{\bf s}}\}_{{\bf s}}$ from $\{e_{{\bf s}}\}_{{\bf s}}$.
In fact we can obtain  even a closed--form expression.

\begin{corollary}\label{th:ALGORITHM3}
With the notations of Theorem \ref{th:ALGORITHM} one has
$d_{{\bf s}}=\min_{{\bf s}\leq \widetilde{{\bf s}}\leq {\bf m}} \{ |\widetilde{{\bf s}}-
{\bf s}|+e_{\widetilde{\bf s}}\}$  for any ${\bf 0}\leq {\bf s}\leq {\bf m}$.
(Here $|{\bf s}|=\sum_{v,k} s_{v,k_v}$.)
In particular,
$$d_Z(l')=d_{{\bf 0}}=\min_{{\bf 0}\leq {\bf s}\leq {\bf m}} \{ |{\bf s}|+e_{{\bf s}}\}.$$
(By the end of \ref{bek:algsetup} one also has
$\min_{{\bf s}\leq \widetilde{{\bf s}}\leq {\bf m}} \{ |\widetilde{{\bf s}}-
{\bf s}|+e_{\widetilde{\bf s}}\}=
\min_{{\bf s}\leq \widetilde{{\bf s}}} \{ |\widetilde{{\bf s}}-
{\bf s}|+e_{\widetilde{\bf s}}\}$  and
$\min_{{\bf 0}\leq {\bf s}\leq {\bf m}} \{ |{\bf s}|+e_{{\bf s}}\}
=\min_{{\bf 0}\leq {\bf s}} \{ |{\bf s}|+e_{{\bf s}}\}$.)
\end{corollary}
\begin{proof}
By Theorem \ref{th:ALGORITHM}{\it (1)} for any $\widetilde{{\bf s}}\geq {\bf s}$ one has
$d_{{\bf s}}-d_{\widetilde{\bf s}}\leq |\widetilde{{\bf s}}-{\bf s}|$, and by (\ref{eq:6.2b})
$d_{\widetilde{\bf s}}\leq e_{\widetilde{\bf s}}$. These two imply  $d_{{\bf s}}\leq
|\widetilde{{\bf s}}-{\bf s}|+e_{\widetilde{\bf s}}$, hence
$d_{{\bf s}}\leq \min_{{\bf s}\leq \widetilde{{\bf s}}\leq {\bf m}} \{ |\widetilde{{\bf s}}-
{\bf s}|+e_{\widetilde{\bf s}}\}$.
Next we show that   $d_{{\bf s}}$ in fact equals
$|\widetilde{{\bf s}}-{\bf s}|+e_{\widetilde{\bf s}}$ for some $\widetilde{{\bf s}}$. The wished
$\widetilde{{\bf s}}$ is the last term of the sequence $\{{\bf s}_i\}_{i=0}^t$ constructed as follows.
Set ${\bf s}_0:={\bf s}$. Then, assume that ${\bf s}_i$ is already constructed, and that there exists
$(v,k)$ such that $d_{{\bf s}_i}=d_{({\bf s}_i)^{v,k}}+1$. Then set ${\bf s}_{i+1}:=({\bf s}_i)^{v,k}$
(for one of the choices of such  possible $(v,k)$).
This inductive construction will stop after finitely many steps (since each $d_{{\bf s}}\geq 0$).
But if $d_{{\bf s}_t}=d_{({\bf s}_t)^{v,k}}$ for all $(v,k)$, then by \ref{th:ALGORITHM}{\it (2)}
$d_{{\bf s}_t}=e_{{\bf s}_t}$. Hence $e_{{\bf s}_t}=d_{{\bf s}_t}=d_{{\bf s}}-|{\bf s}_t-{\bf s}|$.
\end{proof}

\section{The second algorithm for the computation of $\dim {\rm Im} (c^{l'}(Z))$}\label{bek:algsetup2}

\subsection{Preparation}\label{ss:4.1}
The algorithm from the previous section determines the dimensions  of the Abel maps $d_Z(l')$
  in terms of a finite collection of invariants
of type $e_Z(l')$ associated with a finite sequence of
resolutions obtained via certain extra blowing ups  from $\tX$.
Though, in principle,  $e_Z(l')$ is much simpler than $d_Z(l')$ (it is the
`stabilizer' of $d_Z(l')$), the algorithm is still
slightly cumbersome, it is more theoretical, it is not easy  to apply in
concrete examples: one needs to know all the integers $\{e_{{\bf s}}\}_{{\bf s}}$, that is,
 cf. Proposition \ref{prop:AZ},
all the integers $\{h^1(\calO_{Z_{{\bf s}}|_{\calv_{{\bf s}}\setminus I_{{\bf s}}}}\}_{{\bf s}}$ associated with the tower of blowing ups.   (However,  it is a  necessary intermediate
step in the proof of the new algorithm).

The new algorithm  is considerably  simpler, e.g. it can be formulated in terms of the
resolution $\tX$ (see also the comments below).
It provides $d_Z(l')$   in terms of the filtration $\{\calG_l\}_l$ of 2--forms.

As a starting point, consider the construction from \ref{bek:algsetup}.
For any ${\bf s}$ define the cycle $l_{{\bf s}}\in L$ of $\tX$ by
$$l_{{\bf s}}:=\min \Big\{
\sum_{v\in \calv} \, \min_{1\leq k_v\leq a_v}\{{\bf s}_{v, k_v}\}E_v, Z\Big\}\in L.$$
Set $\calG_{{\bf s}}:=\calG_{l_{{\bf s}}}$ and $g_{{\bf s}}:=\dim \calG_{{\bf s}}$ as well.
Note that (via pullback) there is an inclusion
$\calG_{{\bf s}}\subset \Omega_{Z_{{\bf s}}}(I_{{\bf s}})$.
Indeed, if the pole order of certain $\omega$ along $E_v$ is $\leq {\bf s}_{v,k_v}$ then its pullback
along $F_{v, k_v, {\bf s}_{v, k_v}}$ has no pole.
Hence $g_{{\bf s}}\leq \dim \Omega_{Z_{{\bf s}}}(I_{{\bf s}})= h^1(\calO_Z)-e_{{\bf s}}$ too (cf.
(\ref{eq:6.2b})).
 In particular,
\begin{equation}\label{eq:ws}
d_{{\bf s}}\leq e_{{\bf s}}\leq h^1(\calO_Z)-g_{{\bf s}}.
\end{equation}
However, in principle it can happen that for a certain $\omega$ with even higher pole than $l_{{\bf s}}$
its pullback is in
$\Omega_{Z_{{\bf s}}}(I_{{\bf s}})$. E.g., if $\omega$ in some local coordinates $(u,v)$
of an open set $U$ is $vdu\wedge dv/u^o$ (and $U\cap E=\{u=0\}$) then its pullback via blowing up
(once)  at  $u=v=0$ has pole order $o-2$. This phenomenon can happen even if we blow up a generic point:
 imagine a family of forms $\omega_t$
with `moving divisor', parametrized by $t$ given by $(v-t) du\wedge dv/u^o$. Then, even if we blow up $E$ at a generic point $u=v-t_0=0$, in the family $\{\omega_t\}_t$ there is a form
$\omega_{t_0}$ whose pole along $E_v$ is $o$ while its pullback has pole $o-2$. Hence
the equality of subspaces  $\calG_{{\bf s}}\subset \Omega_{Z_{{\bf s}}}(I_{{\bf s}})$,
or of the equality $e_{{\bf s}}= h^1(\calO_Z)-g_{{\bf s}}$ in principle is subtle and it is
hard to test.

Note also that the invariant $h^1(\calO_Z)-g_{{\bf s}}$
conceptually (and technically) is much simpler than $e_{{\bf s}}$. E.g., it depends only on
$v\mapsto \min _{k_v\leq a_v}\{ {\bf s}_{v,k_v}\}$, and it can be described  via a cycle of $\tX$
(namely $l_{{\bf s}}$)
instead of the geometry of the tower $\tX_{{\bf s}}$. Nevertheless, via the next theorem,  it
still contains sufficient information to determine $d_{{\bf s}}$, in particular $d_Z(l')$. In order
to emphasize the parallelism between the two algorithms we formulate them in a completely symmetric way
(in particular, the first parts are completely identical).

\begin{theorem}\label{th:ALGORITHM4} {\bf (Second algorithm)}
 With the above notations  the following facts hold.

(1) $ d_{{\bf s}} -  d_{{\bf s}^{v, k}} \in \{0, 1\}$.

(2) If for some fixed ${\bf s}$ the numbers $\{d_{{\bf s}^{v, k}}\}_{v,k}$ are not the same,
then $d_{{\bf s}} = \max_{v, k}\{\,d_{{\bf s}^{v, k}}\}$.
 In the case when all the numbers $\{d_{{\bf s}^{v, k}}\}_{v,k}$ are the same,
 then if this common value $d_{{\bf s}^{v, k}}$ equals $h^1(\calO_Z)-g_{{\bf s}}$, then $d_{{\bf s}} =
h^1(\calO_Z)-g_{{\bf s}}  =d_{{\bf s}^{v, k}}$;  otherwise $d_{{\bf s}} = d_{{\bf s}^{v, k}}+1$.
\end{theorem}

For the proof  see section \ref{s:twisted}.

\begin{corollary}\label{cor:formula2} With the notations of \ref{ss:4.1} and of Theorem
\ref{th:ALGORITHM4}, for
 $l' \in - S'$ and $Z\geq E$ one has
\begin{equation}\label{eq:form2}
d_Z(l') = \min_{{\bf s}}\{\, |{\bf s}| + h^1(\calO_Z) - g_{{\bf s}}\,\}.
\end{equation}
\end{corollary}
The proof runs similarly as the proof of Corollary \ref{th:ALGORITHM3}.

The formula (\ref{eq:form2}) can be rewritten in a different flavour.
\begin{corollary}\label{cor:formula3}
For $l' \in - S'$ and $Z\geq E$ one has
\begin{equation}  \label{eq:form3}
d_Z(l') = \min_{0\leq Z_1 \leq Z}\{\, (l', Z_1) + h^1(\calO_Z) - h^1(\calO_{Z_1})\, \}.
\end{equation}
\end{corollary}
\begin{proof}
From \ref{bek:filtrforms} $g_{{\bf s}}=\dim \calG_{{\bf s}}=h^1(\calO_{l_{{\bf s}}})$ and also
$|{\bf s}|\geq \sum _va_v (l_{{\bf s}})_v=(l',l_{{\bf s}})$, and $0\leq l_{{\bf s}} \leq Z$,
 hence
$ \min_{{\bf s}}\{\, |{\bf s}| + h^1(\calO_Z) - g_{{\bf s}}\,\}\geq
 \min_{0\leq Z_1 \leq Z}\{\, (l', Z_1) + h^1(\calO_Z) - h^1(\calO_{Z_1})\, \}$.
 The opposite inequality is also true since
 any such $Z_1$ can be represented as  a certain $l_{{\bf s}}$ with
 $|{\bf s}|=(l',l_{{\bf s}})$.
\end{proof}
\begin{example}\label{ex:dom} (1) {\bf ($c^{l'}(Z)$ constant)} \ For any
$0\leq Z_1\leq Z$ one has  $ (l', Z_1)\geq 0$ and  $h^1(\calO_Z) \geq  h^1(\calO_{Z_1})$, hence
$d_Z(l')=0$ happens exactly when there exists $Z_1$ with
$ (l', Z_1) + h^1(\calO_Z) - h^1(\calO_{Z_1})=0$, or, $ (l', Z_1)=0$ and
 $h^1(\calO_Z) = h^1(\calO_{Z_1})$. This means that $Z_1\leq Z|_{\calv\setminus I}$, where $I$ is the
 $E^*$--support of $l'$, a fact which
 (together with  $h^1(\calO_Z) = h^1(\calO_{Z_1})$)
 implies $h^1(\calO_Z) = h^1(\calO_{Z|_{\calv\setminus I}})$ too. Hence, $d_Z(l')=0$ if and only if
 $h^1(\calO_Z) = h^1(\calO_{Z|_{\calv\setminus I}})$. This is exactly the statement of
 \cite[6.3(v)]{NNI}.

(2) {\bf $c^{l'}(Z)$ is dominant} if and only if $d_Z(l')=h^1(\calO_Z)$, hence, via (\ref{eq:form3}),
if and only if $h^1(\calO_{Z_1})\leq (l',Z_1)$  for any $0\leq Z_1\leq Z$.
This can be seen in a different way as follows.
First, if $c^{l'}(Z)$ is dominant, then, for any $0<Z_1\leq Z$,  $c^{l'}(Z_1)$ is dominant too, hence
 $(l', Z_1) = \dim(\eca^{l'}(Z_1) ) \geq  \dim(H^1(\calO_{Z_1}))$. Conversely, if
$ (l',Z_1)\geq h^1(\calO_{Z_1})$ and $Z_1 > 0$ then $(l',Z_1)-h^1(\calO_{Z_1})>-h^0(\calO_{Z_1})$,
that is, $\chi(-l')<\chi(-l'+Z_1)$, hence $c^{l'}(Z)$ is dominant by \cite[Thm. 4.1.1]{NNI}, cf.
\ref{ss:DOMINANT} here. Note that the characterization \ref{ss:DOMINANT} for dominant property is topological.

(3) By (\ref{eq:form3}) {\bf  $\im (c^{l'}(Z))$ is a hypersurface} if and only if
$\min_{0\leq Z_1 \leq Z} \{(l',Z_1)-h^1(\calO_{Z_1})\}=-1$. Since $h^0(\calO_{Z_1})\geq 1$, this implies
that
$\chi(-l')=\min_{0\leq l\leq Z} \chi(-l' +l)$.

The  converse statement is not true: take e.g. a Gorenstein elliptic singularity
with length of elliptic sequence $m+1$. (For elliptic singularities consult \cite{weakly,NNIII,NNtop}.
For more on the Abel map of elliptic singularities see \cite{NNIII}.)
Set $Z\gg 0$ and $-l'=Z_{min}$, the fundamental (minimal)
cycle. Then $\im (c^{l'}(Z))=1$ and $h^1(Z)=p_g=m+1$. However,
$\chi(Z_{min})=\min_{0\leq l\leq Z} \chi(Z_{min} +l)=0$. Therefore, if $m=1$ then
$\im (c^{l'})$ is a hypersurface, but for $m\geq 2$ it is not. It is instructive to
consider with the same topological data (elliptic numerically
Gorenstein singularity with $m\geq 1$, $Z\gg 0$,
$-l'=Z_{min}$) the generic analytic structure. Then $p_g=1$ (cf. \cite{Laufer77,NNII}) but
$\im (c^{l'}(Z))$ is a point (this follows from part {\it (1)} too). Hence
$\im (c^{l'}(Z))$ is a hypersurface for any $m\geq 1$. In particular, the property that
$\im (c^{l'}(Z))$ is a hypersurface  is not a topological property.

\end{example}
\begin{example}\label{ex:superisol} {\bf (Superisolated singularities)} \
Assume that $(X,o)$ is a hypersurface superisolated singularity whose link is a rational homology sphere. More precisely, $(X,o)=\{F(x_1,x_2,x_3)=0\}$, where the homogeneous terms $F_i$ of $F$ are as follows:
$\{F_d=0\}$ defines an irreducible rational cuspidal curve in ${\mathbb C}{\mathbb P}^2$ and
$\{F_{d+1}=0\}\cap {\rm Sing}\{F_d=0\}$ is empty in ${\mathbb C}{\mathbb P}^2$. (For details see
\cite{Ignacio,LMNsi,NNI}.) Consider the minimal good resolution and let $E_0$ be the irreducible exceptional curve corresponding to $C$ (the exceptional curve of the first blow up of the maximal ideal).
Assume that $l'=-kE_0^*$ for some $k\geq 1$ and $Z\geq Z_K$. For any
${\bf m}=(m_1,m_2,m_3)\in \Z_{\geq 0}^3$ write $|{\bf m}|=\sum_im_i$.
Then by the discussion from
\cite[11.2]{NNI}  one has the following facts: $p_g=d(d-1)(d-2)/6=\#\{{\bf m}: |{\bf m}|\leq d-3\}$,
this is exactly the cardinality  of the
 set of forms of type ${\bf x}^{{\bf m}}\omega$, where $\omega$ is the Gorenstein form.
 The pole order of  $\omega$ along $E_0$ is $d-2$, and the vanishing order of ${\bf x}^{{\bf m}}$ along $E_0$ is $|{\bf m}|$. $\{{\bf x}^{{\bf m}}\omega\}_{{\bf m}}$
 constitute a basis in
 $H^0(\Omega^2_{\tX}(Z))/H^0(\Omega^2_{\tX})$. Hence, for $0\leq s\leq d-2$ one has  $g_{s}= \dim  \calG_{sE_0}=
\#\{{\bf m}: d-2-s\leq |{\bf m}|\leq d-3\}$  and  $h^1(\calO_Z)-g_s=\binom{d-s}{3}$. 
In particular,
$$d_{Z}(-kE^*_0)=\min_{0\leq s\leq d-2}\ \{ks+\textstyle{\binom{d-s}{3}}\}.$$
In \cite[11.2]{NNI} $d_Z(-kE^*_0)$ was computed in a different way as $ \sum_{j=0}^{d-3}
\min\{k, \binom{j+2}{2}\}$. The identification of the two numerical answers is left to the reader.
(Use $\sum_{j=0}^t\binom{j+2}{2}=\binom{t+3}{3}$.)
\end{example}
\begin{example}\label{ex:wh}
For {\bf weighted homogeneous germs} (and $l'=-kE^*_0$, where $E_0$ is the central vertex of the star shaped graph) $d_Z(l')$ was computed by a similar method in \cite[\S 12]{NNI}.
\end{example}


\begin{remark}\label{rem:rem}
 (1) In Theorems \ref{th:ALGORITHM} and \ref{th:ALGORITHM4} (and Corollaries
\ref{th:ALGORITHM3} and \ref{cor:formula2} as well) the functions ${\bf s}\mapsto e_{{\bf s}}$
and ${\bf s}\mapsto h^1(\calO_Z)-g_{{\bf s}}$ serve as `test--functions':
``if this common value $d_{{\bf s}^{v, k}}$ equals the test value, then $d_{{\bf s}} =
d_{{\bf s}^{v, k}}$,   otherwise $d_{{\bf s}} = d_{{\bf s}^{v, k}}+1$''.
Via this fact in mind, the second algorithm is rather surprising: the test function for each
fixed $v$ depends only on ${\bf s}\mapsto \min _{0\leq k_v\leq a_v}s_{v,k_v}=(l_{{\bf s}})_v$,
hence does not depend on the number of integers
$\{s_{v,k_v}\}_{0\leq k_v\leq a_v}$, or, on $a_v$. However,
the final output, namely $d_{{\bf s}}$  (and the right hand side of (\ref{eq:form2})    and
the algorithm itself)
do depend on $l'$.  We  encourage the reader to work out the algorithm for an example
when $a_v\geq 2$ (say, for $-l'=2E_v^*$).

(2)
Notice that the formulas $  \min_{{\bf s}}( |{\bf s}| + h^1(Z) - g_{{\bf s}})$ and $\min_{{\bf s}}( |{\bf s}| + e_{{\bf s}})$ can be defined without any restriction on the numbers $g_{{\bf s}}$ and $e_{{\bf s}}$,
however in our case these numbers are restricted.
For example we have $ \min_{{\bf s} \geq {\bf s}_1}( |{\bf s}| - |{\bf s}_1| + h^1(Z) - g_{{\bf s}}) - \min_{{\bf s} \geq {\bf s}_1^{v, k}}( |{\bf s} \geq {\bf s}_1^{v, k}| +  h^1(Z) - g_{{\bf s}}) \in \{0, 1\}$ for all $v, k, {\bf s}_1$.
Or, $g_{{\bf s}} \leq |{\bf s}|$ for all ${\bf s}$ if and only if $\chi(-l') < \chi(-l'+ l)$ for all $Z \geq l > 0$ (cf. Example \ref{ex:dom}(2)).

(3) {\bf (Bounds for ${\rm codim}\, \im \, c^{l'}(Z)$)} \
In some expression the codimension of $\im ( c^{l'}(Z))$
 appears more naturally. E.g., we have the following two
general statements from \cite[Prop. 5.6.1]{NNI} (under the conditions of  Corollary \ref{cor:formula3}):

(a)  $h^1(Z,\calL)\geq {\rm codim}\, \im( c^{l'}(Z))$ for any
$\calL\in \im( c^{l'}(Z))$.
Equality holds whenever $\calL$ is generic in $\im( c^{l'}(Z))$.

(b) ${\rm codim}\, \im \, c^{l'}(Z)\geq \chi(-l')- \min_{0\leq l\leq Z} \chi(-l' +l)$,
and this inequality is strict whenever $c^{l'}(Z)$ is not dominant.
(This can be compared with the discussion from Example \ref{ex:dom}(3).)

Note that Corollary \ref{cor:formula3} reads as:
\begin{equation}\label{eq:codim}
{\rm codim}\, \im ( c^{l'}(Z))=
\max_{0\leq Z_1\leq Z}\, \{ \, h^1(\calO_{Z_1})-(l',Z_1)\,\}.
\end{equation}
\end{remark}


\bekezdes Before we state the next theorem let us emphasise the obvious  fact
that  for any
$0\leq Z_1\leq Z$ the natural restriction (linear projection) $r:H^1(\calO_{Z})\to
H^1(\calO_{Z_1})$ is surjective, hence for any irreducible constructible
subset $C_1\subset H^1(\calO_{Z_1})$  one has $\dim r^{-1}(C_1)-\dim C_1=h^1(\calO_{Z})-
h^1(\calO_{Z_1})$.

However, though the restriction of $r$ to $\im (c^{l'}(Z))\to \im (c^{l'}(Z_1))$ is dominant, in general $\dim \im (c^{l'}(Z))$ can be smaller than $\dim r^{-1}(\im (c^{l'}(Z_1)))$.


\bekezdes
It is instructive to see that certain  extremal geometric phenomenons
(indexed by effective cycles) are realized by the very same set of cycles.
\begin{lemma}\label{lem:MINSETS}
The following three sets of cycles coincide (for fixed $Z\geq E$ and $l'\in-\calS'$ as above):

(I) the set of cycles $Z_1$ with $0\leq Z_1\leq Z$ realizing the minimality in (\ref{eq:form3}), that is: $d_Z(l')=(l',Z_1)+h^1(\calO_Z)-h^1(\calO_{Z_1})$.

(II) the set of cycles $Z_1$ with $0\leq Z_1\leq Z$ such that
 {\it (i)} the map
$\eca^{l'}(Z) \to H^1(Z_1)$ is birational onto its image,
and {\it (ii) }the generic fibres of the restriction of $r$,
$r^{im}:\im(c^{l'}(Z)) \to \im(c^{l'}(Z_1))$,  have dimension $h^1(\calO_{Z})-h^1(\calO_{Z_1})$.
(That is, the fibers of $r^{im}$   have maximal possible dimension.)

(III)  the set of cycles $Z_1$ with $0\leq Z_1\leq Z$ such that  for the generic element
$\calL_{gen}^{im}\in\im (c^{l'}(Z))$  and arbitrary section $s\in H^0(Z_1, \calL_{gen}^{im})_{reg}$ with divisor $D$
{\it (i)} in the (analogue of the
Mittag-Lefler sequence associated with the exact sequence $0\to \calO_{Z_1}\stackrel{\times s}{\longrightarrow} \calL_{gen}^{im} \to \calO_D\to 0$, cf. \cite[3.2]{NNI}),
$$0\to H^0(\calO_{Z_1})\stackrel{\times s}{\longrightarrow} H^0(Z_1,\calL_{gen}^{im})
\to \bC^{(Z_1,l')}
\stackrel{\delta}{\longrightarrow} H^1(\calO_{Z_1})\to h^1(Z_1,\calL_{gen}^{im})\to 0$$
$\delta$ is injective, and {\it (ii)} $h^1(Z,\calL_{gen}^{im})=h^1(Z_1,\calL_{gen}^{im})$.
\end{lemma}
\begin{proof}
For (I)$\Rightarrow$(II) use the following.
First recall that $\dim \eca^{l'}(Z')=(l',Z')$
for any effective cycle $Z'$. Next, from  (\ref{eq:form3}),
 there exists an effective cycle $Z_1 \leq Z$, such that
$\dim \im(c^{l'}(Z))=  (l', Z_1) + h^1(\calO_Z) - h^1(\calO_{Z_1})$.
But  $\dim(\im(c^{l'}(Z_1))) \leq \dim \eca ^{l'}(Z_1)=(l', Z_1)$ (cf. \ref{ss:AbelMap})
and $\dim( \im(c^{l'}(Z))) - \dim( \im(c^{l'}(Z_1))) \leq h^1(\calO_Z) - h^1(\calO_{Z_1})$.
Hence, necessarily  we have equalities in both these inequalities.
(I)$\Leftarrow$(II) is similar.

For (II)(i)$\Leftrightarrow$(III)(i) use the fact that $\delta$ is the tangent application
$T_D\im c^{l'}(Z_1)$ at $D$, cf. \cite[3.2]{NNI}, and for (II)(ii)$\Leftrightarrow$(III)(ii)
use  Remark \ref{rem:rem}(3)(a).
\end{proof}

\subsection{Structure theorem for the Abel map}\label{ss:5.1}

The geometric interpretation from Lemma \ref{lem:MINSETS}(II) has the following consequence.

\begin{theorem}\label{th:structure} {\bf (Structure  theorem)}
Fix a resolution $\tX$, a cycle $Z\geq  E$  and a   Chern class $l'\in -\calS'$
as above.

(a) There exists an effective cycle $Z_1 \leq Z$, such that: {\it (i)} the map
$\eca^{l'}(Z) \to H^1(Z_1)$ is birational onto its image,
and {\it (ii) }the generic fibres of the restriction of $r$,
$r^{im}:\im(c^{l'}(Z)) \to \im(c^{l'}(Z_1))$,  have dimension $h^1(\calO_{Z})-h^1(\calO_{Z_1})$.
(Cf. Lemma \ref{lem:MINSETS}(II).)

(b) In particular, for any such $Z_1$, the space $\im (c^{l'}(Z))$ is birationally
equivalent with an affine fibration with affine fibers of dimension
$h^1(\calO_{Z})-h^1(\calO_{Z_1})$ over $\eca^{l'}(Z_1)$.

(c) The set of effective cycles $Z_1$ with property as in {\it (a)} has a unique
minimal and a unique maximal element
denoted by $C_{min}(Z, l')$ and $C_{max}(Z, l')$.
Furthermore, $C_{min}(Z, l')$ coincides with the cohomology cycle of the pair
$(Z,\calL_{gen}^{im})$ (the unique minimal element of the set $\{0\leq Z_1\leq Z\,:\,
h^1(Z,\calL_{gen}^{im}) =h^1(Z_1,\calL_{gen}^{im})$) for the generic $\calL_{gen}^{im}\in
\im ( c^{l'}(Z))$.
\end{theorem}
\begin{proof}  {\it (a)} Use Lemma \ref{lem:MINSETS}.

{\it (c)} Assume that two cycles $Z_1$ and $Z_2$ satisfy {\it (a)}. We claim that
 $Z' := \max\{Z_1, Z_2\}$  satisfies too.

 First, for any cycle $Z''$ with $Z_1\leq Z''\leq Z$, if $Z_1$ satisfies {\it (a)(ii)}
 then $Z''$ satisfies too.
This applies for $Z'$ too. To prove {\it (a)(i)} for $Z'$,
let us  denote by $\eca^{l'}(Z'')_0 \subset \eca^{l'}(Z'')$ the set of divisors whose
support is disjoint from the singular  points of $E$.
If $l'=\sum_va_vE_v^*$ then $\eca^{l'}(Z)_0=\prod _v \eca^{a_vE^*_v}(Z)_0$.
 Using this fact one shows
that the product
$\eca^{l'}(Z')\to \eca^{l'}(Z_1)\times \eca^{l'}(Z_2)$
of the two restrictions  $ \eca^{l'}(Z') \to \eca^{l'}(Z_j)$ ($j=1,2$)
 is birational onto its image (BioIm).
This composed with the product of the maps  $\eca^{l'}(Z_1) \to H^1(Z_1)$ and $\eca^{l'}(Z_2) \to H^1(Z_2)$
(both BioIm) guarantees  that  $\eca^{l'}(Z') \to  H^1(Z_1)\times H^1(Z_2)$ is BioIm too.
This map writes as the composition
 $\eca^{l'}(Z') \to H^1(Z') \to H^1(Z_1)\times H^1(Z_2)$, hence the first term
 $\eca^{l'}(Z') \to H^1(Z')$   should be  BioIm. Hence the claim and
the existence of $C_{max}(Z, l')$ follows.

%
%

In order to prove the existence of
 $C_{min}(Z, l')$, first we claim that the set of cycles $Z^{ii}$,
which satisfy   {\it (a)(ii)} has a unique minimal element $Z^{ii}_{min}$.
This fact via  Remark \ref{rem:rem}(3)(a) is equivalent with the existence of the (unique)
cohomological cycle for the pair $(Z,\calL_{gen}^{im})$. This was proved in \cite[5.5]{NNI}, see also
\cite[4.8]{MR}.
Next, we claim that
the map $\eca^{l'}(Z^{ii}_{min}) \to H^1(Z^{ii}_{min})$ is BioIm as well.
From the existence of the cycle $C_{max}(\cdot, l')$  (already proved above), applied for
$Z^{ii}_{min}$, there exists a cycle $C_{max}(Z^{ii}_{min}, l')\leq Z^{ii}_{min}$,
which satisfies {\it (a)}. In particular, {\it (a)(ii)} is valid for the pair
$C_{max}(Z^{ii}_{min}, l')\leq Z^{ii}_{min}$. By the definition of
$Z^{ii}_{min}$ the condition  {\it (a)(ii)} is valid for the pair
$Z^{ii}_{min}\leq Z$ too. Hence,  {\it (a)(ii)} is valid for the pair
$C_{max}(Z^{ii}_{min}, l')\leq Z$ as well.
Therefore, by
the definition of $Z^{ii}_{min}$ necessarily $C_{max}(Z^{ii}_{min}, l')= Z^{ii}_{min}$,
hence $Z^{ii}_{min}$ satisfies {\it (a)}.
\end{proof}

\section{Example. The case of generic analytic structure}\label{ss:GEN}
\subsection{}
Let us fix the topological type of a good resolution of a normal surface singularity, and we assume
that the analytic type on $\tX$ is generic (in the sense of \cite{NNII}, see \cite{LaDef1} as well).
Recall that in such a situation, if $Z'=\sum n_vE_v$ is a non--zero  effective cycle, whose support
$|Z'|=\cup_{n_v\not=0}E_v$ is connected, then by \cite[Corollary 6.1.7]{NNII} one has
$$h^1(\calO_{Z'})=1-\min _{|Z'|\leq l\leq Z',\ l\in L}\ \{\chi(l)\}.$$
\begin{cor} \label{cor:GEN}
Assume that $\tX$ has a generic analytic type, $Z \geq E$ an integral  cycle and  $l' \in -S'$.
For any $0\leq Z_1\leq Z$  write $E_{|Z_1|}$ for $\sum _{E_v\subset |Z_1|}E_v$. Then
\begin{equation}\label{eq:gen}
d_{Z}(l') = 1  -
 \min_{E \leq l \leq Z}\ \{ \chi(l)\} +  \min_{0\leq Z_1\leq Z}\big\{ \, (l', Z_1)  +  \min_{E_{|Z_1|} \leq l \leq Z_1 }\ \{ \chi(l)\}
  - \chi(E_{|Z_1|})
 \, \big\}.
\end{equation}
In particular, $d_{Z}(l')=\dim (\im c^{l'}(Z)) $ is topological.
\end{cor}
Let us  concentrate again on the {\it codimension} $h^1(\calO_Z)-d_Z(l')$
of $\im (c^{l'}(Z))\subset \pic^{l'}(Z)$ instead of the
dimension. Then, (\ref{eq:gen}) reads as
\begin{equation}\label{eq:gen2}
{\rm codim}\, \im( c^{l'}(Z))=  \max_{0\leq Z_1\leq Z}\big\{ \, -(l', Z_1)
- \min_{E_{|Z_1|} \leq l \leq Z_1 }\ \{ \chi(l)\}
 +  \chi(E_{|Z_1|})
 \, \big\}.
\end{equation}
This is a rather complicated combinatorial expression in terms of the intersection lattice $L$. The next lemma aims to simplify it.
\begin{proposition}\label{lem:minsimpl} Consider the assumptions of  Corollary \ref{cor:GEN}.
Let $Z_1$ be minimal such that the maximum in (\ref{eq:gen2}) is realized for it. Then
$\min_{E_{|Z_1|} \leq l \leq Z_1 }\ \{ \chi(l)\}=\chi(Z_1)$. In particular,
\begin{equation}\label{eq:gen3}
{\rm codim}\, \im (c^{l'}(Z))=  \max_{0\leq Z_1\leq Z}\big\{ \, -(l', Z_1)
-  \chi(Z_1) +  \chi(E_{|Z_1|}) \, \big\}.
\end{equation}
The maximum at the right hand side is realized e.g. for the cohomology cycle of
$\calL^{im}_{gen}\in \im (c^{l'}(Z))\subset \pic^{l'}(Z)$.
Furthermore,
\begin{equation}\label{eq:gen4}
h^1(Z, \calL)\geq  \max_{0\leq Z_1\leq Z}\big\{ \, -(l', Z_1)
-  \chi(Z_1) +  \chi(E_{|Z_1|}) \, \big\}
\end{equation}
for any $\calL\in \im (c^{l'}(Z))$ and equality holds for generic $\calL_{gen}^{im}\in
 \im (c^{l'}(Z))$.
\end{proposition}
\begin{proof} Assume that the minimum $\min_{E_{|Z_1|} \leq l \leq Z_1 }\ \{ \chi(l)\}=\chi(Z_1)$
is realized by some $l_1$. Then $(l',Z_1)\geq (l',l_1)$ (since $l' \in -\calS'$),
$ \min_{E_{|Z_1|} \leq l \leq Z_1 }\{\chi(l)\}=\min_{E_{|l_1|} \leq l \leq l_1 }\{\chi(l)\}$
 and
$\chi(E_{|Z_1|})=\chi(E_{|l_1|})$ hence  $-(l', Z_1)
-  \min_{E_{|Z_1|} \leq l \leq Z_1 }\ \{ \chi(l)\}  +  \chi(E_{|Z_1|})\leq
-(l',l_1)-  \min_{E_{|l_1|} \leq l \leq l_1 }\ \{ \chi(l)\} +\chi(E_{|l_1|})$.
Since the maximality in (\ref{eq:gen2}) is realized by $Z_1$,  which is minimal with this property,
necessarily $Z_1=l_1$. Next,
$$
 \max_{0\leq Z_1\leq Z}\big\{-(l', Z_1)- \min_{E_{|Z_1|} \leq l \leq Z_1 }\ \{ \chi(l)\}
 +  \chi(E_{|Z_1|})\}\geq  \max_{0\leq Z_1\leq Z}\big\{-(l', Z_1)
- \chi(Z_1) +  \chi(E_{|Z_1|})\}.$$
But the maximum at the left hand side is realized by a term from the right.

For the last statement use again  Remark \ref{rem:rem}(3)(a).
\end{proof}

\subsection{} The identity (\ref{eq:gen3}),
valid for a generic analytic structure of $\tX$, extends to an optimal inequality
valid for {\it any analytic structure}.

\begin{theorem}\label{th:tz}
Consider an {\em arbitrary}  normal surface singularity $(X,o)$, its resolution $\tX$, $Z\geq E$ and
$l'\in -\calS'$. Then
${\rm codim}\, \im (c^{l'}(Z))=h^1(Z, \calL^{im}_{gen})$ (cf. Remark \ref{rem:rem}(3)(a))
satisfies
\begin{equation}\label{eq:gen5}
{\rm codim}\, \im (c^{l'}(Z))\geq  \max_{0\leq Z_1\leq Z}\big\{ \, -(l', Z_1)
-  \chi(Z_1) +  \chi(E_{|Z_1|}) \, \big\}.
\end{equation}
In particular, for any $\calL\in \im (c^{l'}(Z))$ one also has (everything computed in $\tX$)
\begin{equation}\label{eq:gen5b}
h^1(Z,\calL)\geq h^1(Z,\calL_{gen}^{im})={\rm codim}\, \im (c^{l'}(Z))\geq  \max_{0\leq Z_1\leq Z}\big\{ \, -(l', Z_1)
-  \chi(Z_1) +  \chi(E_{|Z_1|}) \, \big\}.
\end{equation}
\end{theorem}

Note that the right hand side of (\ref{eq:gen5}) is a sharp  topological lower bound for
${\rm codim}\, \im (c^{l'}(Z))$. The inequality (\ref{eq:gen5}) can also be interpreted as the semi-continuity statement
$${\rm codim}\, \im( c^{l'}(Z))(\mbox{arbitrary analytic structure})\geq
{\rm codim}\, \im( c^{l'}(Z))(\mbox{generic analytic structure}).$$
\begin{proof}
Consider the identity (\ref{eq:codim}) applied for an arbitrary $\tX$ and for the generic
$\tX$, denoted by $\tX_{gen}$. Then, by semi-continuity of $h^1(\calO_{Z_1})$ with respect
to the analytic structure as parameter space (see e.g. \cite[3.6]{NNII}),
for any fixed effective cycle  $Z_1>0$,  $h^1(\calO_{Z_1})$ computed in $\tX$ is
greater than or equal to $h^1(\calO_{Z_1})$ computed in $\tX_{gen}$. Therefore, by
(\ref{eq:codim}) one has ${\rm codim}\, \im (c^{l'}(Z))(\mbox{in} \
\tX)\geq {\rm codim}\, \im (c^{l'}(Z))(\mbox{in} \ \tX_{gen})$.
Then for $\tX_{gen}$ apply  (\ref{eq:gen3}).
\end{proof}
\begin{remark}
Certain upper bounds
 for $\{h^{1}(Z,\calL)\}_{\calL\in \pic^{l'}(Z)}$, valid for any analytic structure,
 were established in
\cite[Prop. 5.7.1]{NNI} (see alo Remark \ref{rem:h1nagy}). However, an optimal upper bound
is not known (see \cite{NO17} for a particular case). Large $h^1$--values are realized by special
strata, whose existence and study is extremely hard.
\end{remark}

\subsection{The cohomology of $\calL_{gen}^{im}(l)$}\label{ss:shifted}
Assume that $Z\geq E$, $l'\in -\calS'$ and let $\calL^{im}_{gen}$ be a generic element
of $\im (c^{l'}(Z))$. If the analytic structure of $(X,o)$ is generic, then by Proposition
\ref{lem:minsimpl}
 $h^1(Z,\calL^{im}_{gen})=t_Z(l')$, where
$t_Z(l')$ is the topological expression from the right hand side of (\ref{eq:gen3}).

 Our goal is to give a topological lower bound for $h^1(Z,\calL)$, where
$\calL:=\calL^{im}_{gen}(l)=\calL^{im}_{gen}\otimes \calO(l)\in \pic^{l'+l}(Z)$ whenever
$l\in L_{>0}$. In this way we will control the generic element of the `new' strata
$\calO(l)\otimes (\im(c^{l'}(Z)))$ of $\pic^{l'+l}(Z)$,
unreachable directly by the previous result. Our hidden goal is to construct in this way
line bundles with `high' $h^1$.

For simplicity we will assume that all the coefficients of $Z$  are sufficiently large
(even compared with $l$, hence the coefficients of $Z-l$ are large as well).
The monomorphism of sheaves
$ \calL^{im}_{gen}|_{Z-l}\hookrightarrow\calL^{im}_{gen}(l)$ gives
 $h^0(Z-l, \calL^{im}_{gen})\leq h^0(Z, \calL^{im}_{gen}(l))$, hence
 $$h^1(Z-l, \calL^{im}_{gen})+\chi(Z-l, \calL^{im}_{gen})\leq
 h^1(Z, \calL^{im}_{gen}(l))+\chi(Z, \calL^{im}_{gen}(l)).$$
By a computation regarding $\chi$ this transforms into
 $$h^1(Z, \calL^{im}_{gen}(l))\geq h^1(Z-l, \calL^{im}_{gen})+\chi(-l'-l)-\chi(-l').$$
If  $\tX$ is generic and $Z, Z-l\gg 0$ then  $h^1(Z-l, \calL^{im}_{gen})=t_{Z-l}(l')=
t_Z(l')$, hence
\begin{equation}\label{eq:h1tz}
h^1(Z, \calL^{im}_{gen}(l))\geq t_Z(l')-\chi(-l')+\chi(-l'-l).
\end{equation}
E.g., with the choice $l=-l'\in\calS'\cap L_{>0}$ we get that
$\calL^{im}_{gen}(-l')\in \pic^0 (Z)$ and
\begin{equation}\label{eq:h1tz2}
h^1(Z, \calL^{im}_{gen}(-l'))\geq t_Z(l')-\chi(-l').
\end{equation}
\begin{remark}\label{rem:h1nagy}
By \cite[Prop. 5.7.1]{NNI} for $Z\gg 0$, $\calL\in\pic(Z)$ with  $c_1(\calL)\in -\calS'$  one has
$h^1(Z, \calL)\leq p_g$ whenever either $H^0 (Z, \calL)=0$ or $\calL \in \im(c^{l'}(Z))$. For other line bundles
a weaker bound is established (see [loc. cit.]), which does not guarantee $h^1(\calL)\leq p_g$.
However, it is not so easy  to find singularities and
bundles with $h^1(\calL)>p_g$ in order to show that such cases indeed might appear.
In the next \ref{ex:tzineq} we provide such an examples (with a recipe to find many others as well) based partly on (\ref{eq:h1tz2}).
\end{remark}

\begin{example}\label{ex:tzineq}
Assume that we can construct a nonrational
 resolution graph which satisfies the following (combinatorial)
 properties, valid for certain $Z\gg 0$
and  $l'\in -\calS'\cap L$:
\begin{equation}\label{eq:req}
\begin{split}
(a) \ \ &\ t_Z(l') \geq \chi(-l') -\min_{l\geq 0} \chi(-l'+l) +2, \ \mbox{and }\\
(b) \ \ &-l' \leq   \max \calm, \ \mbox{where $\calm:=\{l\in L_{>0}: \chi(l)=\min \chi\}$}.
\end{split}
\end{equation}
Now, if we consider the generic analytic structure supported on this  topological type, then
$\min_{l\geq 0} \chi(-l'+l)\stackrel{(b)}{=}\min \chi=1-p_g$ (for the second identity
use \cite[Cor. 5.2.1]{NNII}), hence
$t_Z(l')-\chi(-l')\stackrel{(a)}{\geq} -1 +p_g+2=p_g+1$. This combined with
 (\ref{eq:h1tz2}) gives  $h^1(Z, \calL^{im}_{gen}(-l'))>p_g$.

Next we show that (\ref{eq:req}) can be realized.
Consider two copies $\Gamma_1$ and $\Gamma_2$ of the following graph

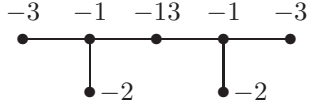
\begin{picture}(300,45)(30,0)
\put(125,25){\circle*{4}}
\put(150,25){\circle*{4}}
\put(175,25){\circle*{4}}
\put(200,25){\circle*{4}}
\put(225,25){\circle*{4}}
\put(150,5){\circle*{4}}
\put(200,5){\circle*{4}}
\put(125,25){\line(1,0){100}}
\put(150,25){\line(0,-1){20}}
\put(200,25){\line(0,-1){20}}
\put(125,35){\makebox(0,0){$-3$}}
\put(150,35){\makebox(0,0){$-1$}}
\put(175,35){\makebox(0,0){$-13$}}
\put(200,35){\makebox(0,0){$-1$}}
\put(225,35){\makebox(0,0){$-3$}}
\put(160,5){\makebox(0,0){$-2$}}
\put(210,5){\makebox(0,0){$-2$}}
\end{picture}

The wished graph $\Gamma$ consists of $\Gamma_1$, $\Gamma_2$ and a new vertex $v$, which has two adjacent
edges connecting $v$ to the  $(-13)$-vertices of  $\Gamma_1$ and  $\Gamma_2$.
Let the decoration of $v$ be $-b_v$ where $b_v\gg0$. One verifies that
the minimal cycle is $Z_{min}=(b_v-2)E_v^*$, whose $E_v$--multiplicity is 1. We set $-l':= Z_{min}$.
Since $\max\calm\in \calS_{an} \subset \calS'\cap L$ (cf. \cite[5.7]{NNII})
we get that $-l'=Z_{min}\leq \max\calm$.
One verifies that $\chi(Z_{min})=-3$ (e.g. by Laufer's criterion), and also that $\min\chi=-5$
(realized e.g. for $2Z_{min}-E_v$).
Therefore $\chi(-l')-\min _{l\geq 0} \chi(-l'+l)+2= -3+5+2=4$. On the other hand,
the expression (under max) in (\ref{eq:gen3}) for
$Z_1=Z_{min}(\Gamma_1)+Z_{min}(\Gamma_2)$ supported on
 $\Gamma\setminus v$ is 4, hence $t_Z(l')\geq 4$.
\end{example}

\section{Appendix. Geometrical aspects behind the lower bound Theorem \ref{th:tz}
}\label{s:tz}

\subsection{} Let us discuss with more details the geometry behind the inequality
(\ref{eq:gen5}). Along the discussion we will  provide a second independent proof of it and we
also provide several examples, which show its sharpness/weakness  in several situations.
Similar construction (with similar philosophy) will appear in forthcoming manuscripts on the subject as well. The construction of the present section shows also in a conceptual way how one can produce
different sharp lower bounds for sheaf  cohomologies (for another case see e.g. subsection
\ref{ss:cohtwisted}).

We provide the new proof in several steps. First, we define a topological lower bound for
${\rm codim}\,\im (c^{l'}(Z))$, which (a priori) will have a more elaborated form then
the right hand side $t_Z(l')$ of (\ref{eq:gen5}).
Then via several steps we will simplify it
and we show that in fact it is exactly $t_Z(l')$.


\begin{definition}\label{def:D}
For any $Z>0$ with $|Z|$ connected we define $D(Z,l')$
as 0 if $c^{l'}(Z)$ is dominant and  1 otherwise. (For a criterion see \ref{ss:DOMINANT}.)
Furthermore, set
\begin{equation}\label{def:T}
T(Z,l'):=\chi(-l')-\min_{0\leq l\leq Z, l\in L} \chi (-l'+l)+D(Z,l').
\end{equation}
\end{definition}
 By  \cite[Theorem 5.3.1]{NNI}  for any singularity $(X,o)$, any resolution $\tX$, any
$Z>0$ and  $l'\in L'$, and for  $\calL_{gen}$ generic in  $\pic^{l'}(Z)$ one has
\begin{equation}\label{eq:INEQ1}
h^1(Z,\calL_{gen})=\chi(-l')-\min_{0\leq l\leq Z, l\in L} \chi (-l'+l).
\end{equation}

By \cite[Prop. 5.6.1]{NNI}, see also \ref{rem:rem}(3),
for any $Z\geq E$ and for any $l'\in -\calS'$,
if $\calL_{gen}^{im}$ is a generic element of $\im (c^{l'}(Z))$,  then
$h^1(Z,\calL_{gen}^{im})={\rm codim} \, \im (c^{l'}(Z))$ satisfies (the semicontinuity)
\begin{equation}\label{eq:INEQ2}
h^1(Z,\calL_{gen}^{im})\geq
\chi(-l')-\min_{0\leq l\leq Z, l\in L} \chi (-l'+l)+D(Z,l')
= h^1(Z,\calL_{gen})+D(Z,l')=T(Z,l').
\end{equation}
\begin{remark}\label{rem:restrict}
Assume that $Z>0$ is a nonzero cycle with connected support $|Z|$, but with $Z\not\geq E$.
Then the statements  from
(\ref{eq:INEQ2}) remain valid for such $Z$ once we replace $l'$ by its restriction $R(l')$, where
$R:L'\to L'(|Z|)$  is the natural cohomological operator dual to the natural
homological inclusion $L(|Z|)\hookrightarrow L$.
(For this apply the statement for the singularity supported on $|Z|$.)
On the other hand, for $l\in L(|Z|)$ one has
$\chi(-R(l'))-\chi(-R(l')+l)=-\chi(l)-(R(l'),l)_{L(|Z|)}=-\chi(l)-(l',l)=\chi(-l')-\chi(-l'+l)$.
Hence, in fact,  (\ref{eq:INEQ2}) remains valid in its original form  for any such
$Z>0$ with $|Z|$ connected.
\end{remark}

\begin{example}\label{ex:Dnagy}
The difference $h^1(Z,\calL_{gen}^{im})-h^1(Z,\calL_{gen})$ can be arbitrary large. Indeed,
let us start with a singularity with an arbitrary  analytic structure,
we fix a resolution $\tX$ with dual graph
$\Gamma$, and we  distinguish a vertex, say  $v_0$,
 associated with the irreducible divisor $E_0$. Let $k$ ($k>0$)  be the number of connected components
of  $\Gamma\setminus v_0$, and we assume that each of them is non--rational.
Furthermore, we choose  $Z\gg 0$, hence $h^1(\calO_Z)=p_g$.
Let $\tX|_{\calv\setminus v_0}$ be a small neighbourhood of $\cup_{v\not=v_0}E_v$,
let $\{\tX_i\}_{i=1}^k$ be its connected components, and set $p_{g,i}=h^1(\calO_{\tX_i})$
for the geometric genus of the singularities obtained from $\tX_i$ by collapsing its exceptional curves.
Write also $\Gamma\setminus v_0=\cup_i \Gamma_i$.
We also assume that  $-l'=nE^*_0$ with $n\gg0$.

Since $n$ is large, $\im (\widetilde{c}^{l'}(Z))=A_Z(l')$, hence
$d_Z(l')=e_Z(l')=p_g-\sum_i p_{g,i}$, cf. \cite[Th. 6.1.9]{NNI} or Theorem \ref{prop:AZ} here.
Hence, cf. (\ref{eq:INEQ2}), ${\rm codim}(\im \widetilde{c}^{l'}(Z))=h^1(\calO_Z)-d_Z(l')=
h^1(Z,\calL_{gen}^{im})=\sum_i p_{g,i}$ (in particular, $\widetilde{c}^{l'}$ is not dominant).

Next we compute $h^1(Z,\calL_{gen})=\chi(nE_0^*)-\min _{l\geq 0} \chi(nE^*_0+l)$. Write $l$ as $l_0E_0
+\widetilde{l}$, where $\widetilde{l}$ is supported on $\cup_{v\not=v_0}E_v$. Then
$\chi(nE_0^*)-\chi(nE^*_0+l)=-\chi(l)-nl_0$. If $l_0=0$ then $-\chi(l)=-\chi(\widetilde{l})$, and its maximal value is $M:=\sum_i (-\min \chi(\Gamma_i))$.
 On the other hand, if $l_0>0$ then for
$n>-M-\min \chi$ one has $-\chi(l)-l_0n<M$.
Hence $h^1(Z,\calL_{gen})=\chi(nE_0^*)-\min _{l\geq 0} \chi(nE^*_0+l)=\sum_i (-\min \chi(\Gamma_i))$.

Now, $p_{g,i}\geq 1-\min\chi(\Gamma_i)$ (cf. \cite{Wa70} or \cite{NNII}),
hence  $h^1(Z,\calL_{gen}^{im})-h^1(Z,\calL_{gen})\geq k$.
\end{example}

\bekezdes
We wish to estimate $h^1(Z,\calL_{gen}^{im})$.
Note that the estimate given by
(\ref{eq:INEQ2}), that is, $h^1(Z,\calL_{gen}^{im})\geq T(Z,l')$, sometimes is week, see the previous example. However, surprisingly, if we replace $Z$ by a smaller cycle $Z'\leq Z$, then we might get a better bound.
More precisely, first note that if $\calL_{gen}^{im}$ is a generic element of $\im (c^{l'}(Z))$, and
$0< Z'\leq Z$, then its restriction $r(\calL_{gen}^{im})$ (via $r:\pic^{l'}(Z)\to \pic^{R(l')}(Z')$)
is a generic element of $\im (c^{l'}(Z'))$. If $Z'$ has more connected components, $Z'=\sum_i Z'_i$
(where each $|Z_i'|$ is connected and $|Z_i'|\cap|Z_j'|=\emptyset $ for $i\not=j$), then for each $Z'_i$
we can apply (\ref{eq:INEQ2}). Therefore, we get
\begin{equation}\label{eq:T}
h^1(Z,\calL_{gen}^{im})\geq h^1(Z',r(\calL_{gen}^{im}))=\sum _i h^1(Z'_i,r(\calL_{gen}^{im}))\geq
\sum_i T(Z'_i, l').\end{equation}
Define
\begin{equation}\label{eq:deft}
t(Z,l'):= \max_{0<Z'\leq Z}\, \sum_i T(Z'_i, l')=\max_{0<Z'\leq Z}\, \Big(
 \, \sum_i (\chi(-l')-\min _{0\leq l_i\leq Z'_i}\chi(-l'+l_i)+D(Z_i',l'))\, \Big).
\end{equation}
(Here there is no need to restrict $l'$, cf. Remark \ref{rem:restrict}.)
Hence (\ref{eq:T}) reads as
\begin{equation}\label{eq:Z'}
 h^1(Z,\calL_{gen}^{im})\geq t(Z,l').\end{equation}
In this estimate the point is the following: though
$\sum_i (\chi(-l')-\min _{0\leq l_i\leq Z'_i}\chi(-l'+l_i)=
\chi(-l')-\min _{0\leq l\leq Z'}\chi(-l'+l)$ is definitely not larger  than
$\chi(-l')-\min _{0\leq l\leq Z}\chi(-l'+l)$, the number of components of $Z'$ might be large,
and the sum of the `non-dominant' contribution terms
$\sum_i D(Z_i',l')$ might increase the right hand side of (\ref{eq:Z'}) --- compared with $T(Z,l')$ --- drastically.
\begin{example}\label{ex:Dnagy2} {\bf (Continuation of Examle \ref{ex:Dnagy})}
The last computation of Example \ref{ex:Dnagy} shows that the maximum of
$\chi(nE_0^*)-\min _{l\geq 0} \chi(nE^*_0+l)$ is obtained for
 $l_0=0$ and $T(Z,l')=1+\sum_i(-\min\chi(\Gamma_i))$. Hence, taking $Z'=\sum_iZ'_i$,
 each $Z'_i$ supported on $\Gamma_i$ and large,  we get that the restriction of $l'$ is zero and
 $\sum _i T(Z'_i, l')=\sum_i(1-\min\chi(\Gamma_i))= T(Z,l')+k-1$.

Summarized (also from Example \ref{ex:Dnagy}),  for any analytic type one has
  $\sum_i p_{g,i}=h^1(Z,\calL_{gen}^{im})\geq t(Z,l')\geq \sum _i T(Z'_i, l')=\sum_i(1-\min\chi(\Gamma_i))$.
   However, if $\tX$ is generic then $p_{g,i}=1-\min \chi(\Gamma_i)$
 (cf. \cite{NNII}), hence, all the inequalities  transform into  equalities.
Hence, for generic analytic structure $h^1(Z,\calL_{gen}^{im})=t(Z,l')$, that is,
(\ref{eq:Z'}) provides  the optimal sharp topological lower bound.

Note also that  both  $t(Z,l')$ and
$\sum_i(1-\min\chi(\Gamma_i))$ are topological, hence if they agree for $\tX$ generic, then they
are in fact equal. 
Since  $p_{g,i}-1+\min \chi(\Gamma_i)$ for arbitrary analytic type can be
 considerably large,  for  {\it arbitrary} analytic types  the inequality
 (\ref{eq:Z'}) can be rather week.
\end{example}

%
%

\subsection{}\label{ss:Z'min} Our goal is to simplify
the expression (\ref{eq:deft}) of $t(Z,l')$.

First we analyse the set of cycles $Z'$ for which the maximum in the
 right hand side of (\ref{eq:deft})
can be realized.  E.g.,
if  $c^{l'}(Z)$ is dominant (equivalently, $t(Z,l')=0$, cf. \ref{ss:DOMINANT})
then any $0\leq Z'\leq Z$ realizes the maximum 0
(with all $l_i=0$). (Indeed, use the fact that $D(Z_2,l')\geq D(Z_1,l')$
for $Z_2\geq Z_1$ and $|Z_i|$ connected.)

In the next Lemmas \ref{lem:Degy} and \ref{prop:tz}
we will assume that $c^{l'}(Z)$ is not dominant.

\begin{lemma}\label{lem:Degy}
(a) Assume that $Z'$ is a  minimal cycle (or a cycle with minimal number of connected components)
among those cycles which realize the maximum in the right
hand side of (\ref{eq:deft}). Then
$D(Z'_i,l')=1$ for all $i$.

(b) If $D(Z'_i,l')=1$ then the minimal value
$\min_{0\leq l_i\leq Z'_i} \chi(-l'+l_i)$ can be  realized by $l_i>0$.
\end{lemma}
\begin{proof}
{\it (a)} Otherwise,  $c^{l'}(Z_i')$ is dominant,
and  by \ref{ss:DOMINANT} $\chi(-l') - \min_{0 \leq l_i \leq Z_i'} \chi(-l' + l_i)=0$
(realized for $l_i=0$).
Hence $T(Z_i',l')=0$, that is, the right hand side  of (\ref{eq:deft})
is realized by
$Z'-Z'_i$ too, contradicting the minimality of $Z'$. {\it (b)} If the wished minimum is realized
by $l_i=0$, and {\it only}  by $l_i=0$, then by \ref{ss:DOMINANT}
$c^{l'}(Z_i')$ is dominant, contradicting $D(Z'_i,l')=1$.
\end{proof}

\begin{example}\label{ex:endvertices}  Though in Example \ref{ex:Dnagy}
we have shown that $h^1(Z, \calL_{gen}^{im})=t(Z,l')$ can be much larger than
$T(Z,l')$ (that is,  the maximizing $Z'$ usually should be necessarily strict smaller than $Z$),
 in some cases $Z'=Z$ still works. Indeed, we claim that
$$\mbox{
if the $E^*$--support
$I$ of $l'$ is included in the set of end vertices of $\Gamma$, then
$t(Z,l')=T(Z,l')$.}$$
%
Let $Z'$ be a cycle for minimal number $n$ of
connected components $\{Z_i'\}_{i=1}^n $ for which the right hand side of
(\ref{eq:deft}) is realized.   We claim that $n=1$.
Indeed, by Lemma \ref{lem:Degy},  each $D(Z'_i,l')=1$.
Let $l_i$ be a  cycle  which realizes
 $\chi(-l') - \min_{0 \leq l \leq Z_i'} \chi(-l' + l)$. By Lemma \ref{lem:Degy}
we can assume  $l_i\not=0$.

  If $n > 1$ then  let $Z_1$ and $Z_2$ be two adjacent component, which means, that there is a vertex
  $u \in |Z_1'|$ and $v \in |Z_2'|$ and a (minimal) path
$u_1 = u, u_2, \cdots , u_t = v$, such that $u_2, \cdots, u_{t-1} \notin |Z'|$ and $u_k$ and $u_{k+1}$ are neighbours in the resolution graph. Moreover, define a new cycle by
$ Z_{1, new}' = Z_1' + Z_2' + \sum_{2 \leq k \leq t-1} E_{u_k}$ and $Z'_{new} = Z_{1, new}'
 + \sum_{3 \leq i \leq n} Z_i'$.
Similarly, let us have a minimal path
between $|l_1|$ and $|l_2|$: vertices  $w_1, \cdots,  w_l$, such that $w_1 \in |l_1|$ and $w_l \in |l_2|$, $w_2, \cdots, w_{l-1} \notin |l_1| \cup |l_2|$ and $w_k , w_{k+1}$ are
neighbours in the resolution graph.  Then define
$l_{1, new} = l_1 + l_2 + \sum_{2 \leq k \leq l-1}E_{w_k}$. The point is that
the vertices $w_2, \cdots, w_{l-1}$ are not end vertices, in particular
 $(l', \sum_{2 \leq k \leq l-1}E_{w_k})=0$.

Note also that $D(Z'_{1,new}, l')=1$.
Then a computation gives that
\begin{equation}
\chi(-l') -  \chi(-l' + l_{1, new}) + D(Z_{1,new}',l')\geq
T(Z_1, l') + T(Z_2, l'),
\end{equation}
or,
$T(Z_{1, new}, l') \geq T(Z_1, l') + T(Z_2, l')$,
contradicting the minimality of $Z'$. Hence necessarily $n=1$.

On the other hand, if $Z'$ is connected, then $T(Z',l')\leq T(Z,l')$, 
 hence
the maximal value in the right hand side of (\ref{eq:Z'}) is realized for $Z$ as well
(and maybe by several other smaller cycles too; here we minimalized $\#|Z'|$ by increasing $Z'$).

The present example together with  Examples \ref{ex:Dnagy}
and \ref{ex:Dnagy2} show that the structure of possible
cycles $Z'$ for which the maximality in (\ref{eq:deft}) realizes can be rather subtle.
\end{example}

\begin{lemma}\label{prop:tz}
Assume that $Z'$ is a  minimal cycle
among those cycle which realizes the maximum in the right
hand side of (\ref{eq:deft}). Then the following facts hold:

(a) $\min _{0\leq l_i\leq Z_i'} \chi (-l'+l_i)$ is realized by $l_i=Z_i'$.

(b)  $\min _{0\leq l_i\leq Z_i'} \chi (l)$ is realized by $l_i=Z_i'$.

(c) $t(Z',l')=t(Z,l')=\sum_i\big( -(Z_i',l')-\chi(Z_i')+1\big)$.


\end{lemma}
\begin{proof}
{\it (a)} For each $Z_i'$ let $l_i$ be minimal non--zero cycle (cf. Lemma \ref{lem:Degy})
such that $M_i:=\chi(-l')-\min _{0\leq l\leq Z_i'} \chi(-l'+l)$ is realized by $l_i$. Let
$l_i=\cup_k l_{i,k}$ be its decomposition into cycles with $|l_{i,k}|$ connected and disjoint.
Since  $M_i=-\chi(l_i)-(l',l_i)\geq 0$, there exists $k$ such that
$\chi(-l')- \chi(-l'+l_{i,k})=-\chi(l_{i,k})-(l',l_{i,k})\geq 0$,
hence  by the criterion from \ref{ss:DOMINANT}  the Abel map $c^{l'}(l_{i,k})$ must be
non--dominant. Thus (using also $D(Z_i',l')=1$ from Lemma \ref{lem:Degy}{\it (a)})
\begin{equation}\label{eq:dag}
\sum_k T(l_{i,k},l')\geq \chi(-l')-\chi(-l'+l_i)+1=T(Z'_i,l').
\end{equation}
 In particular, by the minimality of $Z_i'$, $Z_i'=l_i$.

{\it (b)} By part {\it (a)}  $\chi(Z_i')+(Z_i',l')\leq \chi(l_i)+(l_i,l')$
for any $0\leq l_i\leq Z_i'$.
But, since $l'\in -\calS'$, $(Z_i',l')\geq (l_i,l')$, hence $\chi(Z_i')\leq \chi(l_i)$ for any $0\leq l_i\leq Z_i'$. Part {\it (c)} follows from (\ref{eq:deft}) and  {\it (a)}.
\end{proof}
Recall that in \ref{ss:shifted} we defined $t_Z(l'):=
 \max_{0\leq Z'\leq Z}\big\{ \, -(l', Z')
-  \chi(Z') +  \chi(E_{|Z'|}) \, \big\}$.

\begin{corollary}\label{cor:1} $t(Z,l')=t_Z(l')$.
\end{corollary}
\begin{proof}
If $c^{l'}(Z)$ is dominant then both sides are zero. Otherwise, by Lemma \ref{prop:tz}(c)
(with its notations) $t(Z,l')=\sum_i \big( -(Z_i',l')-\chi(Z_i')+1\big)\leq t_Z(l')$. On the other hand, let us fix some $Z'=\cup_iZ'_i$ for which the maximum in $t_Z(l')$ is realized.
Then we can assume that each $c^{l'}(Z'_i)$ is not dominant. Then
$ -(Z_i',l')-\chi(Z_i')+1=\chi(-l')-\chi(-l'+Z_i')+1\leq
\chi(-l')-\min _{0\leq l_i\leq Z_i'}\chi(-l'+l_i)+D(Z'_i,l')$. Hence
$t_Z(l')\leq t(Z,l')$ too.
\end{proof}

\begin{remark} The second proof of  Theorem \ref{th:tz} follows from
(\ref{eq:Z'}) and Corolary \ref{cor:1}.
\end{remark}

\section{The
$\calL_0$--projected Abel map}\label{s:projAbel}

In this section we introduce a new object, a modification of the Picard group $\pic(Z)$, which
will play a key role in the cohomology computation of the shifted line bundles of type
$\{\calL_0\otimes \calL\}_{\calL\in \im (c^{l'}(Z))}$.

\subsection{The $\calL_0$--projected Picard group}\label{ss:diagrams}
Let $(X,o)$ be a normal surface singularity. For simplicity
we  assume (as always in this manuscript) that the link is a rational homology sphere.
Let $\tX$ be one of its good resolutions and $Z\geq E$ an effective cycle. Fix also $\calL_0\in
\pic(Z)$ such  that $H^0(Z,\calL_0)_{reg}\not=\emptyset$ (cf. \ref{ss:AbelMap}).   Choose
$s_0\in H^0(Z,\calL_0)_{reg}$ arbitrarily, and write ${\rm div}(s_0)=D_0\in \eca ^{l_0'}(Z)$,
where $l'_0=c_1(\calL_0)\in -\calS'$. Motivated by the exponential exact sequence of sheaves
$0\to \Z_{Z}\stackrel {i}\to \calO_Z\to \calO_Z^*\to 0$, we define $\calL_0^*:=
{\rm coker} ( \Z_Z\stackrel {i}{\to}\calO_Z\stackrel{s_0}{\longrightarrow}\calL_0)$, where
the second morphism is the multiplication by (restrictions of) $s_0$. Then we have the following commutative diagram of sheaves:
$$\begin{array}{ccccccccc}
 &&&& 0 && 0 && \\
 &&&& \downarrow  && \downarrow && \\
0 & \longrightarrow & \Z_Z & \stackrel{i}{\longrightarrow} &
\calO_Z  & \longrightarrow & \calO_Z^* & \longrightarrow & 0 \\
 & & \downarrow\vcenter{%
 \rlap{=}} &  & \downarrow\vcenter{%
 \rlap{$\scriptstyle{s_0}$}} & &
 \downarrow\vcenter{%
 \rlap{$\scriptstyle{s_0^*}$}} & & \\
 0 & \longrightarrow & \Z_Z & \longrightarrow &
\calL_0 & \longrightarrow & \calL_0^* & \longrightarrow & 0 \\
& & &  & \downarrow & &
 \downarrow & & \\
 & & &  & \calO_{D_0} & $=$  & \calO_{D_0} & & \\
 & & &  & \downarrow & &
 \downarrow & & \\
 & & &  & 0 &   & 0 & &
 \end{array} $$
where $s_0^*$ is induced by $s_0$. At cohomological level  we get the (identical/renamed) diagrams
$$\begin{array}{cccccccccccccc}
 & H^0(\calO_{D_0}) & =  & H^0(\calO_{D_0})  & & &
  & \hspace{2cm}
 & H^0(\calO_{D_0}) & =  &H^0(\calO_{D_0})  && &  \\
 & \downarrow \vcenter{%
 \rlap{$\delta^0$}} && \downarrow\vcenter{%
 \rlap{$\delta$}} &&&
  &
   & \downarrow \vcenter{%
 \rlap{$\delta^0$}} && \downarrow\vcenter{%
 \rlap{$\delta$}} &&&
 \\
0 \to & H^1(\calO_Z) & \to &
H^1(\calO^*_Z)  & \stackrel{c_1}{\to}  & L' & \to   0
&
0 \to & {\rm Pic}^0(Z) & \to &
\pic(Z)  & \stackrel{c_1}{\to}  & L' & \to   0 \\
 & \downarrow \vcenter{%
 \rlap{$s^0$}} && \downarrow\vcenter{%
 \rlap{$s$}} && \downarrow \vcenter{%
 \rlap{$=$}}&
  &
   & \downarrow \vcenter{%
 \rlap{$s^0$}} && \downarrow\vcenter{%
 \rlap{$s$}} && \downarrow \vcenter{%
 \rlap{$=$}}   &  \\
 0  \to & H^1(\calL_0) & \to &
H^1(\calL^*_0)  & \stackrel{c_1}{\to}  & L' & \to  0
&
0 \to & {\rm Pic}^0_{\calL_0}(Z) & \to &
\pic_{\calL_0}(Z)  & \stackrel{c_1}{\to}  & L' & \to  0 \\
 & \downarrow && \downarrow &&&
  &
   & \downarrow  && \downarrow&&& \\
 & 0&& 0 &&&
  &
   & 0  && 0 &&&
 \end{array} $$
where we use the notation $\pic_{\calL_0}(Z):=H^1(Z, \calL^*_0)$ --- and call it {\it the
$\calL_0$--projected Picard group} ---,   and (its linearization)
$\pic^0_{\calL_0}(Z):=H^1(Z, \calL_0)$.
Note that the classical first Chern class map $c_1$ factorizes to a well--defined map
$c_1:\pic_{\calL_0}(Z)\to L'$.
Set also $\pic^{l'}_{\calL_0}(Z):=c_1^{-1}(l')$ for any $l'\in L'$; it is an affine space isomorphic to $\pic^{l'}(Z)/ \im (\delta)$
 associated with the vector space $\pic^0_{\calL_0}(Z)=H^1(Z, \calL_0)=H^1(\calO_Z)/\im (\delta^0)$.

The corresponding vector spaces appear in the following exact sequences as well.
Let us take another line bundle $\calL\in \pic^{l'}(Z)$ without fixed components,
$s\in H^0(Z,\calL)_{reg}$ and $D:={\rm div}(s)$. Then one can take the exact sequences
$0\to \calO_Z\stackrel{s}{\to} \calL\to \calO_D\to 0$ and
$0\to \calL_0\stackrel{s}{\to } \calL_0\otimes \calL\to \calO_D\to 0$. They induce
(at cohomology, or `tangent' vector space level)   the following commutative diagram
$$\begin{array}{cccccc}
& & H^0(\calO_{D_0}) & =  & H^0(\calO_{D_0})  &  \\
& & \downarrow \vcenter{%
 \rlap{$\delta^0$}} && \downarrow & \\
H^0(\calO_D) &  \stackrel{\delta^0_{\calL}}{\to} & H^1(\calO_Z) & \stackrel{s}{\to} &
H^1(\calL)  & \to   0\\
 \Big\downarrow \vcenter{%
 \rlap{=}} && \Big\downarrow\vcenter{%
 \rlap{$s^0_{\calL_0}$}} && \Big\downarrow& \\
H^0(\calO_D) &  \stackrel{\bar{\delta}^0_{\calL}}{ \to} & H^1(\calL_0) & \stackrel{s}{\to} &
H^1(\calL_0\otimes \calL)  & \to   0\\
& & \downarrow  && \downarrow & \\
& & 0  && 0 &
 \end{array} $$
This is related with the Abel map $c^{l'}(Z):\eca^{l'}(Z)\to \pic^{l'}(Z)$
 as  follows. Recall from \cite[3.2.2]{NNI} that the tangent linear map
 $T_D \,c^{l'}(Z):T_D\, \eca^{l'}(Z)\to T_{\calL}\, \pic^{l'}(Z)$ can be identified
 with $\delta^0_{\calL}:H^0(\calO_D)\to H^1(\calO_Z)$. Therefore, if $\calL=\calL_{gen}^{im}$
 is a generic
 element of $\im(c^{l'}(Z))$ then ${\rm codim}\,\im (c^{l'}(Z))=\dim \, H^1(\calO_Z)/ \im (\delta^0_{\calL})= h^1(Z, \calL)$.
 Similarly, consider the composition
 $$c^{l'}_{\calL_0}(Z):\eca^{l'}(Z)\,\stackrel{ c^{l'}(Z)}{\longrightarrow}\,
  \pic^{l'}(Z)\, \stackrel{s^0_{\calL_0}}{\longrightarrow}\,  \pic^{l'}_{\calL_0}(Z).$$
 We call it {\it the $\calL_0$--projection  of the Abel map $c^{l'}(Z)$}.
Using the previous paragraph we obtain that the tangent linear map
 $T_D \,c^{l'}_{\calL_0}(Z):T_D\, \eca^{l'}(Z)\to T_{\calL}\, \pic^{l'}_{\calL_0}(Z)$ can be identified
 with $\bar{\delta}^0_{\calL}=s^0_{\calL_0}\circ \delta^0_{\calL}:H^0(\calO_D)\to H^1(\calL_0 )$. Therefore, if $\calL$ is a generic
 element of $\im(c^{l'}_{\calL_0}(Z))$
 (or, it is the image by $s_{\calL_0}$ of a generic element $\calL^{im}_{gen}$ of
  $\im(c^{l'}(Z))$)
 then
 \begin{equation}\label{eq:codimtwisted}
 {\rm codim}\,\im (c^{l'}_{\calL_0}(Z))=\dim \, H^1(\calL_0)/ \im (\bar{\delta}^0_{\calL})
 =h^1(Z, \calL_0\otimes \calL).\end{equation}
 This fact fully motivates the next point of view:
  if one wishes to study $h^1(Z, \calL_0\otimes \calL)$
 with $\calL_0$ fixed and $\calL\in \pic^{l'}(Z)$ then --- as a tool ---
 the right Abel map is the $\calL_0$--projected $c^{l'}_{\calL_0}(Z)$.

\subsection{The cohomology $h^1(Z,\calL_0\otimes \calL)$.}\label{ss:cohtwisted}
Using the exact sequence $H^0(\calO_D)\to H^1(\calO_Z)\stackrel{s}{\to } H^1(Z, \calL)\to 0$ and
$h^0(\calO_D)=(l',Z)$ we obtain the inequality $h^1(Z,\calL)\geq h^1(\calO_Z)-(l',Z)$.
Usually it is not sharp, since $\delta_{\calL}^0$ might not be  injective.
However, as in the prototype construction from section \ref{s:tz} (and even in
its preceding  sections), if we consider any $Z_1\leq Z$
then we also have $h^1(Z,\calL)\geq h^1(Z_1,\calL)\geq h^1(\calO_{Z_1})-(l',Z_1)$, hence
$h^1(Z,\calL)\geq \max_{Z_1\leq Z}\{  h^1(\calO_{Z_1})-(l',Z_1)\}$, and, remarkably, this
for the generic $\calL_{gen}^{im}\in \im(c^{l'}(Z))$ is an equality (cf. (\ref{eq:codim})).

Similarly, using the exact sequence $H^0(\calO_D)\to H^1(Z, \calL_0)\stackrel{s}{\to }
H^1(Z, \calL_0\otimes \calL)\to 0$
 we obtain $h^1(Z,\calL_0\otimes \calL)\geq h^1(\calL_0)-(l',Z)$.
Again, this usually is not sharp. However,  by the same procedure,
\begin{equation}\label{eq:twistedineq}
h^1(Z,\calL_0\otimes \calL)\geq \max_{0\leq Z_1\leq Z}\{  h^1(Z_1,\calL_0)-(l',Z_1)\}.
\end{equation}
In the next section (cf. Corollary \ref{cor:formL})
we will prove that this is again an equality for
   the generic $\calL=\calL_{gen}^{im}\in \im(c^{l'}_{\calL_0}(Z))$.
   (The above inequality (\ref{eq:twistedineq}) can be compared with (\ref{eq:h1tz}) as well.)

\subsection{Compatibility with Laufer duality and differential forms}\label{ss:LauferD}
Consider the perfect pairing
$\langle \,,\,\rangle:H^1(\calO_Z)\otimes H^0(\Omega^2_{\tX}(Z))/H^0(\Omega^2_{\tX})\to \C$ from
\ref{bek:Z}, see alo \cite{NNI}. Once we fix $D_0={\rm div}(s_0)$ of certain $s_0\in H^0(Z,\calL_0)_{reg}$, we can define $\Omega_Z(D_0):= (\im (\delta^0_{\calL_0}))^\perp \subset H^0(\Omega^2_{\tX}(Z))/H^0(\Omega^2_{\tX})$.
It is generated by forms which vanish on
the image of the tangent map
$T_{D_0} \,c^{l'_0}(Z)$,
 identified with $\delta^0_{\calL_0}$, cf. \ref{bek:OmegaD} and (\ref{eq:dual}).
   The pairing $\langle\,,\,\rangle$ induces a perfect pairing
 $\langle\,,\,\rangle_{\calL_0}:H^1(Z, \calL_0)\otimes \Omega_Z(D_0)\to\C$, see also Theorem \ref{th:Formsres}.

 \subsection{The $\calG$--filtration of $\Omega_Z(D_0)=H^1(\calL_0)^*$}\label{ss:filtrL}
 Consider the situation and notations of Definition \ref{bek:filtrforms};
 in particular,  $\calG_l = H^0(\Omega^2_{\tX}(l))/H^0(\Omega^2_{\tX})$ for any $0<l\leq Z$.
  In the presence of $\calL_0=\calO_Z(D_0)$ as above, we have the subspace
  $\Omega_Z(D_0)=(\im \delta^0)^\perp \subset H^0(\Omega^2_{\tX}(Z))/H^0(\Omega^2_{\tX})$,
  and the induced perfect pairing
  $\langle \,,\,\rangle _{\calL_0} : H^1(Z,\calL_0) \otimes \Omega_Z(D_0)\to \C$. Similarly, for any
  $0<l\leq Z$, we have the analogues data
   $\Omega_l(D_0)=(\im (\delta^0|_l))^\perp \subset H^0(\Omega^2_{\tX}(l))/H^0(\Omega^2_{\tX})$,
  and the induced perfect pairing
  $\langle \,,\,\rangle _{\calL_0|_l} : H^1(l,\calL_0) \otimes \Omega_l(D_0)\to \C$.
One has the following inclusions  inside $H^0(\Omega^2_{\tX}(Z))/H^0(\Omega^2_{\tX})$
$$\begin{array}{ccc}
\Omega_l(D_0) & \longrightarrow & \Omega _Z(D_0) \\
\downarrow && \downarrow\\
\calG_l & \longrightarrow & H^0(\Omega^2_{\tX}(Z))/H^0(\Omega^2_{\tX})
 \end{array} $$
and, in fact, $\Omega_l(D_0)=\Omega_Z(D_0)\cap \calG_l$. Hence
$\{\Omega_l(D_0)\}_l=\{\Omega_Z(D_0)\cap \calG_l\}_l$ filters $\Omega_Z(D_0)$. Moreover, by $\langle\,,\,\rangle_{\calL_0|_l}$,
one has
$\dim\, \Omega_Z(D_0)\cap\calG_l=\dim\, \Omega_l(D_0)=h^1(l,\calL_0)$.

\subsection{Dimensions/Notations}\label{ss:dimensions}
 The dimension of $\im (c^{l'}_{\calL_0}(Z))$ is denoted by $d_{\calL_0,Z}(l')$.

If $A_Z(l')$ is the smallest affine space which contains $\im (c^{l'}(Z))$ in
$\pic^{l'}(Z)$, then $s_{\calL_0}(A_Z(l'))$ is the smallest affine space which contains
$\im (c^{l'}_{\calL_0}(Z))$. We denote it by $A_{\calL_0,Z}(l')$ and its dimension by $e_{\calL_0,Z}(l')$.
From definitions
$d_{\calL_0,Z}(l')\leq e_{\calL_0,Z}(l')$.


In the next section we provide two algorithms for the computation of $d_{\calL_0,Z}(l')$,
the analogues of the algorithms from Theorems \ref{th:ALGORITHM} and \ref{th:ALGORITHM4}.

\section{$\calL_0$--projected versions of the algorithms}\label{s:twisted}

\subsection{The setup}\label{ss:setup} Let  us fix $(X,o)$, a good resolution $\tX$, $Z\geq E$ and $l'\in-\calS'$.  We also fix  a line bundle $\calL_0$ as in section \ref{s:projAbel}, whose notations
we will adopt. In order to estimate $d_{\calL_0,Z}(l')$ we proceed as in sections \ref{s:GenAlg}
and \ref{bek:algsetup2}. In particular, we perform the modificatiosn $\pi_{{\bf s}}
: \tX_{{\bf s}}\to \tX$, and we adopt the notations of \ref{bek:algsetup} as well.
By the generic choice of the centers of blow ups we can assume that they differ from the
support of $D_0$.
Notice that we have a natural identification between $H^1(\calO_Z)$ and $H^1(\calO_{Z_{{\bf s}}})$,
and also between  $H^1(\calO^*_Z)$ and $H^1(\calO^*_{Z_{{\bf s}}})$.
Furthermore, we denote the divisor $\pi_{{\bf s}}^{-1}(D_0)$ on $\tX_{{\bf s}}$ still by $D_0$
(basically unmodified),
 and the line bundle $\calO_{Z_{{\bf s}}}(D_0) $ still by $\calL_0$.
 Then we have the identification of $H^0(Z,\calO_D)$ with $H^0(Z_{{\bf s}},\calO_D)$,
 and also $H^1(Z,\calL_0)\simeq H^1(Z_{{\bf s}}, \calL_0)$ and
$H^1(Z,\calL^*_0)\simeq H^1(Z_{{\bf s}}, \calL^*_0)$
(hence identifications of the corresponding commutative diagrams from \ref{ss:diagrams} as well).
The subspace $\Omega_{Z_{{\bf s}}}(D_0)$ in $H^1(\calO_{Z_{{\bf s}}})^*=H^1(\calO_Z)^*$ is also
`stable' of dimension $h^1(Z,\calL_0)$.

Write $d_{\calL_0,{\bf s}}$ and $e_{\calL_0,{\bf s}}$ the corresponding dimensions associated with $\tX_{{\bf s}}$ defined as in \ref{ss:dimensions}.
Then  $d_{\calL_0,{\bf s}} \leq e_{\calL_0,{\bf s}}$.
If  ${\bf s}={\bf 0}$ then
 $d_{\calL_0,{\bf 0}}= d_{\calL_0,Z}(l')$ and
  $e_{\calL_0,{\bf 0}}= e_{\calL_0,Z}(l')$.

\begin{theorem}\label{th:ALGORITHML}

(1) $d_{\calL_0,{\bf s}} - d_{\calL_0,{\bf s}^{v,k}} \in \{0, 1\}$. Moreover,
$d_{\calL_0,{\bf s}} = d_{\calL_0,{\bf s}^{v,k}}$
if and only if for a generic point $\bar{\calL} \in
 \im ( c_{\calL_0}^{l'_{{\bf s}}}(Z_{{\bf s}}))$ the set of divisors in
 $ (c_{\calL_0}^{l'_{{\bf s}}}(Z_{{\bf s}}))^{-1}(\bar{\calL})$
 do not have a base point on $F_{v,k,{\bf s}_{v,k}}$.

(2)    If for some fixed ${\bf s}$ the numbers
$\{d_{\calL_0,{\bf s}^{v,k}}\}_{v,k}$ are not the same, then $d_{\calL_0,{\bf s}} =
\max_{v, k}\{\,d_{\calL_0,{\bf s}^{v, k}}\}$.
 In the case when all the numbers $\{d_{\calL_0,{\bf s}^{v, k}}\}_{v,k}$ are the same,  then if this common value $d_{\calL_0, {\bf s}^{v, k}}$ equals $e_{\calL_0,{\bf s}}$,
 then $d_{\calL_0,{\bf s}} = e_{\calL_0,{\bf s}}  =d_{\calL_0,{\bf s}^{v, k}}$;
 otherwise $d_{\calL_0,{\bf s}} = d_{\calL_0,{\bf s}^{v, k}}+1$.
\end{theorem}

\begin{proof}  {\it (1)} Assume first that either $s_{v,k}\geq 1$ or $a_v=1$. Then
divisors from  $\eca^{l'_{{\bf s}}}(Z_{{\bf s}})$ intersect $F_{v,k,{\bf s}_{v,k}}$ by multiplicity one,
hence the intersection (supporting) point gives a map $q:\eca^{l'_{{\bf s}}}(Z_{{\bf s}})\to F_{v,k,{\bf s}_{v,k}}$, which is dominant. Moreover, $\eca^{l'_{{\bf s}^{v,k}}}(Z_{{\bf s}^{v,k}})$ is birational with a generic fiber of $q$ (the fiber over the point which was blown up), hence
the first statement  follows. Note also that
$ d_{\calL_0,{\bf s}} = d_{\calL_0,{\bf s}^{v, k}}$ if and only if the generic fiber of the
$\calL_0$--projected Abel map
$c_{\calL_0}^{l'_{{\bf s}}}$ is not included in a $q$--fiber.
This implies the second part of {\it (1)}.


If ${\bf s}_{v,k}=0$ and $a_v>1$ then write $l'_-:=l'_{{\bf s}}-E^*_v$ and
consider the `addition map'
$s: \eca ^{E^*_v}(Z_{{\bf s}})\times \eca ^{l'_-}(Z_{{\bf s}})\to
\eca ^{l'_{{\bf s}}}(Z_{{\bf s}})$, which is dominant and quasifinite (cf. \cite[Lemma 6.1.1]{NNI}).
Let $q:\eca ^{E^*_v}(Z_{{\bf s}})\to E_v$ be given by the supporting  point as before. Then if
$q^{-1}(gen)$ is a generic fiber of $q$ (above the point which was blown up), then the restriction of
$s$ to $q^{-1}(gen)\times \eca ^{l'_-}(Z_{{\bf s}})$
with target   $\eca^{l'_{{\bf s}^{v,k}}}(Z_{{\bf s}^{v,k}})$
is dominant and quasifinite.   Hence the arguments can be repeated.

{\it (2)}
First notice that if the numbers $\{d_{\calL_0,{\bf s}^{v, k}}\}$ are not the same
then from {\it (1)} we have
 $d_{\calL_0,{\bf s}} \leq  \min_{v, k}d_{\calL_0,{\bf s}^{v, k}} + 1
 \leq \max_{v, k}d_{\calL_0,{\bf s}^{v, k}}\leq d_{\calL_0,{\bf s}}$, hence
  $d_{\calL_0,{\bf s}} =  \max_{v, k}d_{\calL_0,{\bf s}^{v, k}}$.

 Next, assume that the numbers $\{ d_{\calL_0,{\bf s}^{v, k}}\}$ are the same, say $d$.

If $d_{\calL_0,{\bf s}}=d$
 then part {\it (1)} reads as follows:
 $d_{\calL_0,{\bf s}} = d_{\calL_0, {\bf s}^{v, k}}$ for all $v$ and $k$
 if and only if for a generic $\bar{\calL}\in
 \im ( c_{\calL_0}^{l'_{{\bf s}}}(Z_{{\bf s}}))$
 the set of divisors in
 $ (c_{\calL_0}^{l'_{{\bf s}}}(Z_{{\bf s}}))^{-1}(\bar{\calL})$
 do not have a base point on any of the curves  $\{F_{v,k,{\bf s}_{v,k}}\}_{v,k}$.

Let us choose a generic element $\bar{\calL} \in \im( c_{\calL_0}^{l'_{{\bf s}}}(Z_{{\bf s}}))$, which is in particular a regular value of  $c_{\calL_0}^{l'_{{\bf s}}}(Z_{{\bf s}})$
and the  generic divisors in $\eca^{l'_{{\bf s}}}(Z_{{\bf s}})$ mapped  to $\bar{\calL}$ are
in fact generic divisors of $\eca^{l'_{{\bf s}}}(Z_{{\bf s}})$ itself.

Next, take an element in $\Omega_{Z_{{\bf s}}}(D_0)$ (for details see \ref{ss:LauferD}) represented by a form $\omega$,  such that the class of $\omega $
vanish on $T_{\bar{\calL}}\im (c_{\calL_0}^{l'_{{\bf s}}}(Z_{{\bf s}}))$.

Then choose a generic $D$  from  $\eca^{l'_{{\bf s}}}(Z_{{\bf s}})$, which is mapped to $\bar{\calL}$ and which has no common points with the support of $\omega$ (we can even assume additionally that it is transversal and reduced).  Then we apply the previous statements for $\bar{\calL}:= c_{\calL_0}^{l'_{{\bf s}}}(Z_{{\bf s}})(D)$.

In  particular, the class of $\omega$ vanish on $\im (T_Dc_{\calL_0}^{l'_{{\bf s}}}(Z_{{\bf s}}))$
so $\omega $ cannot have pole along any of the curves  $\{F_{v,k,{\bf s}_{v,k}}\}_{v,k}$,
that is, it belongs to $\Omega_{Z_{{\bf s}}}(I_{{\bf s}})$, cf. Theorem \ref{th:Formsres} and
Lemma \ref{lem:dualInt}. Hence  $d_{\calL_0,{\bf s}} =e_{\calL_0, {\bf s}}$, cf. Lemma
\ref{lem:imcA}, and also
 $d=e_{\calL_0, {\bf s}} $  too.

 On the other hand if  $d=e_{\calL_0, {\bf s}} $,
then from    $d_{\calL_0, {\bf s}^{v, k}}\leq d_{\calL_0, {\bf s}}\leq  e_{\calL_0,{\bf s}}$ we get
$d= d_{\calL_0, {\bf s}}$. Hence    $d_{\calL_0,{\bf s}}=d$ if and only if
$d=e_{\calL_0,{\bf s}} $. Otherwise $d_{\calL_0,{\bf s}}$ should be $d+1$ by {\it (1)}.
\end{proof}

\subsection{Notations for the second algorithm}\label{ss:notII}
Consider the setup of \ref{ss:4.1} and combine it with the one from \ref{ss:setup}, where $\calL_0$
enters in the picture. Accordingly,  we have the following subspaces (inclusions):
$$\begin{array}{ccccccc}
\Omega_{Z_{{\bf s}}}(D_0)\cap \calG_{l_{{\bf s}}} & \to &
\Omega_{Z_{{\bf s}}}(D_0)\cap\Omega_{Z_{{\bf s}}}(I_{{\bf s}}) & \stackrel{j}{\longrightarrow} & \Omega_{Z_{{\bf s}}}(D_0) &=& H^1(Z, \calL_0)^*\\
\downarrow && \downarrow && \downarrow &&\\
\calG_{l_{{\bf s}}} & \to & \Omega_{Z_{{\bf s}}}(I_{{\bf s}}) & \stackrel{i}{\longrightarrow} & H^0(\Omega^2_{\tX_{{\bf s}}}(Z_{{\bf s}}))/ H^0(\Omega^2_{\tX_{{\bf s}}})&=& H^1(\calO_Z)^*
\end{array}$$
The codimension of the inclusion $i$ is $e_{{\bf s}}$ and the dimension of $\calG_{{\bf s}}$ is
$g_{{\bf s}}$ providing the inequality $e_{{\bf s}}\leq h^1(\calO_Z)-g_{{\bf s}}$.
Similarly, the codimension of $j$
 is $e_{\calL_0,{\bf s}}$ and the dimension of
 $\Omega_{Z_{{\bf s}}}(D_0)\cap \calG_{l_{{\bf s}}}$ will be denoted by
$g_{\calL_0, {\bf s}}$ providing the inequality
 $e_{\calL_0, {\bf s}}\leq h^1(Z,\calL_0)-g_{\calL_0, {\bf s}}$. Hence
 \begin{equation}\label{eq:ws2}
  d_{\calL_0, {\bf s}}\leq e_{\calL_0, {\bf s}}\leq h^1(Z,\calL_0)-g_{\calL_0, {\bf s}}.
 \end{equation}
 It is conveninent to lift the ${\bf s}$--independent subspace
 $\Omega_{Z_{{\bf s}}}(D_0)=\Omega_Z(D_0)$ of
 $H^0(\Omega^2_{\tX}(Z))/ H^0(\Omega^2_{\tX})$
 as $\Omega_{\tX}(D_0):=\pi^{-1}(\Omega_Z(D_0))$ by the projection $\pi:H^0(\Omega^2_{\tX}(Z))\to
  H^0(\Omega^2_{\tX}(Z))/ H^0(\Omega^2_{\tX})$.

\begin{theorem}\label{th:ALGORITHM4L} \
(1) $ d_{\calL_0, {\bf s}} -  d_{\calL_0, {\bf s}^{v, k}} \in \{0, 1\}$.

(2) If for some fixed ${\bf s}$ the numbers $\{d_{\calL_0, {\bf s}^{v, k}}\}_{v,k}$ are not the same,
then $d_{\calL_0, {\bf s}} = \max_{v, k}\{\,d_{\calL_0, {\bf s}^{v, k}}\}$.
 In the case when all the numbers $\{d_{\calL_0, {\bf s}^{v, k}}\}_{v,k}$ are the same,
 then if this common value $d_{\calL_0, {\bf s}^{v, k}}$ equals $h^1(Z, \calL_0)-g_{\calL_0,{\bf s}}$, then $d_{\calL_0, {\bf s}} =
h^1(Z,\calL_0)-g_{\calL_0,{\bf s}}  =d_{\calL_0, {\bf s}^{v, k}}$;  otherwise
$d_{\calL_0, {\bf s}} = d_{\calL_0, {\bf s}^{v, k}}+1$.
\end{theorem}

\begin{proof}
Part {\it (1)} was already proved in Theorem \ref{th:ALGORITHML}. Regarding part {\it (2)},
if the numbers $\{d_{\calL_0, {\bf s}^{v, k}}\}$ are not the same then we argue again as in
the proof of Theorem \ref{th:ALGORITHML}.

 Next, assume that the numbers $\{ d_{\calL_0, {\bf s}^{v, k}}\}$ are the same, say $d$.
 Via (\ref{eq:ws2}) and  the first algorithm Theorem
 \ref{th:ALGORITHML}  we need to show that if $d=e_{\calL_0,{\bf s}}$ then necessarily  $d=h^1(Z,\calL_0)-g_{\calL_0, {\bf s}}$ as well.
 However, if $d=e_{\calL_0,{\bf s}}$ then we have $e_{\calL_0,{\bf s}}=
 d_{\calL_0,{\bf s}^{v,k}}$ for all $(v, k)$, hence by
 (\ref{eq:ws2}) we get $e_{\calL_0,{\bf s}}=d=d_{\calL_0, {\bf s}^{v,k}}\leq e_{\calL_0,
 {\bf s}^{v,k}}$. But $e_{\calL_0, {\bf s}}\geq  e_{\calL_0, {\bf s}^{v,k}}$
 by the combination of the argument from (\ref{eq:ineqes}) and the diagram
 from \ref{ss:notII}.  Hence,  $ d_{\calL_0, {\bf s}^{v, k}}=e_{\calL_0, {\bf s}}$ for all $k$
  and $v$
   implies $e_{\calL_0, {\bf s}^{v,k}}= e_{\calL_0, {\bf s}}$ for all $v$ and $k$.

In particular, it is enough to verify the (stronger statement):
\begin{equation}\label{eq:*}\mbox{ if
$e_{\calL_0, {\bf s}^{v,k}}= e_{\calL_0, {\bf s}}$ for all $v$ and $k$ then
 $e_{\calL_0, {\bf s}}= h^1(Z,\calL_0)-g_{\calL_0, {\bf s}}$ as well.}
\end{equation}

Assume that (\ref{eq:*})  is not true, that is,
$e_{\calL_0, {\bf s}^{v,k}}= e_{\calL_0, {\bf s}}$ for all $v$ and $k$, but
 $e_{\calL_0, {\bf s}} <  h^1(Z,\calL_0)-g_{\calL_0, {\bf s}}$.
The last inequality via the diagram from \ref{ss:notII}
says that the inclusion  $\Omega_{Z_{{\bf s}}}(D_0)\cap \calG_{l_{{\bf s}}}\subset
\Omega_{Z_{{\bf s}}}(D_0)\cap\Omega_{Z_{{\bf s}}}(I_{{\bf s}})$ is strict.
This means, that there is a  differential form $\omega\in
\Omega_{\tX}(D_0)$,
with class $[\omega]$ in $H^0(\Omega_{\tX}^2(Z))/H^0(\Omega_{\tX}^2)
\subset H^0(\tX\setminus E,\Omega^2_{\tX})/ H^0(\tX,\Omega^2_{\tX})$,
such that $\omega$ does not have a pole along  the  exceptional
divisor $F_{v, k, {\bf s}_{v, k}}$, however  $[\omega] \notin \calG_{{\bf s}}$.
In particular, there exists
 a vertex $v \in |l'|$, such that the pole order  of
 $\omega$ along  $E_v$ is larger than $(l_{\bf s})_v$.
Notice that this also means $(l_{\bf s})_v = \min_{1 \leq i \leq a_v}{\bf s}_{v, i} < Z_v$.

Let $1 \leq i \leq a_v$ be an integer such that
${\bf s}_{v, i} = (l_{\bf s})_v $ (abridged in the sequel by $t$)
and we  denote the order of vanishing of $\omega$ on an arbitrary exceptional divisor $E_u$ by $b_u$, where
$u$ is an arbitrary vertex along the blowing up procedure.
Next we focus on the string between $v$ and $ w_{v,i,{\bf s}_{v,i}}$ and we denote them by $v_0 = v, \ldots, v_t = w_{v,i,{\bf s}_{v,i}}$.  Set $r:=\min\{ 0 \leq s \leq t \,:\,
 b_{v_s} + t-s \geq 0\}$. Since for $s=t$ one has  $b_{v_t}  \geq 0$ (since
 $\omega$ has no pole along $ F_{v, i, {\bf s}_{v, i}}$) $r$ is well--defined.
On the other hand we have $r\geq 1$. Indeed,
$b_{v_0} + t < 0$, since pole order of $\omega$ along  $E_v$ is  higher than $(l_{\bf s})_v=t$.
Note that  $b_{v_{r-1}} + t-r+1 < 0$ and $b_{v_r}+t-r\geq 0$ imply $b_{v_r}-b_{v_{r-1}}\geq 2$ ($\dag$).

Let $\tX'$ be that resolution obtained from $\tX$,   as an intermediate step of the tower between
$\tX$ and $\tX_{{\bf s}}$,  when in the $(v,i)$ sequence of blow ups we  do not proceed all
${\bf s}_{v,i}$ of them, but we create only the divisors $\{F_{v,i,k}\}_{k\leq r-1}$. Let $\calv'$ be its vertex set and $\{E_u\}_{u\in \calv'}$ its exceptional divisors. On $\tX'$ consider
 the line bundle $\calL:=\Omega^2_{\tX'}(-\sum_{u\in\calv'} b_{u} E_{u})$.
Since $F_{v,i,v_r}$ was  created by blowing up a {\it generic point} $p$ of $E_{v_{r-1}}=F_{v,i,v_{r-1}}$,
the existence of $\omega $ guarantees the existence of a section
 $s \in H^0(\tX', \mathcal{L})$, which does not vanish along $E_{v_{r-1}}$ and
 it has multiplicity $m:=b_{v_r}-b_{v_{r-1}}-1$ at
 the generic point $p\in E_{v_{r1}}$. By ($\dag$) $m\geq 1$.
 By construction, $\omega$ (or $s$) belongs also to the subvectorspace $\Omega_{\tX}(D_0)$ after certain identifications.

Now by the technical Lemma \ref{lem:techn}
(valid for general line bundles, and  separated in section \ref{s:techn})
for any $0 \leq k < m$ and a generic point $p \in E_{v_{r-1}}$ there exists a section $s' \in H^0(\tX', \mathcal{L})$, which
does not vanish along  the exceptional divisor $E_{v_{r-1}}$, and the divisor of $s'$
has  multiplicity $k$ at $p$. We
 apply  for $k = -(b_{v_{r-1}} + t-r+1)-1$. (Note that $0\leq k<m$.)
The section $s'$ gives a differential form $\omega' \in \Omega_{\tX}(D_0)$,
such that if we blow up $E_{v_{r-1}}$ in the generic point $p$ and we denote the new exceptional
divisor by $E_{v_{r, new}}$, then $\omega'$ has wanishing order
$-(t-r +1)$ on $E_{v_{r, new}}$.
This means, that if we blow up it in generic points $t-r+ 1$ times, then $\omega'$ has
 a pole on $E_{v_{t, new}}$, but has no pole on $E_{v_{t + 1, new}}$.
This means that $e_{\calL_0, {\bf s}^{v,i}} \neq e_{\calL_0, {\bf s}}$, which is a contradiction.
\end{proof}
The analogues of Corollaries \ref{cor:formula2} and \ref{cor:formula3} (with similar proofs)
are:
\begin{corollary}\label{cor:formL}
For any $l'\in -\calS'$, $Z\geq E$ and $\calL_0$ with $H^0(Z,\calL_0)_{reg}\not=\emptyset$
one has
$$d_{\calL_0,Z}(l')=\min_{{\bf s}} \{\, |s|+h^1(Z,\calL_0)-g_{\calL_0,{\bf s}}\,\}=
\min_{0\leq Z_1\leq Z} \{\, (l', Z_1) + h^1(Z,\calL_0)-h^1(Z_1,\calL_0)\}.
$$
This combined with (\ref{eq:codimtwisted}) gives for a
 generic $\calL^{im}_{gen}\in \im (c^{l'}(Z))$:
$$h^1(Z,\calL_0\otimes\calL^{im}_{gen})=
\max_{0\leq Z_1\leq Z} \{\, h^1(Z_1,\calL_0)- (l', Z_1)\}.
$$
\end{corollary}
\begin{example}\label{ex:super2}
This is a continuation of Example \ref{ex:superisol} (based on \cite[\S 11]{NNI}), whose notations and statements we will use. Assume that $Z\gg 0$ and $l'=-kE_0^*$ as in \ref{ex:superisol}.
Additionally we take a generic line bundle $\calL_0$ with $c_1(\calL_0)=l'_0=-k_0E_0^*$, $k_0\geq 0$,
(hence $\widetilde{D}_0$ consists of $k_0$ generic irreducible cuts of $E_0$).
Recall that $H^0(\Omega^2_{\tX}(Z))/H^0(\Omega^2_{\tX})$ admits a basis consisting of elements
of type ${\bf x}^{{\bf m}}\omega$, where $\omega$ is the Gorenstein form and $0\leq |{\bf m}|
\leq d-3$. Each `block' $\{|{\bf m}|=j\}$ ($0\leq j\leq d-3$) (which can be identified with
$H^0({\mathbb P}^2, \calO(j))$) contributes with $\binom{j+2}{2}$
monomials. The $k_0$ generic divisors  impose $\min\{k_0, \binom{j+2}{2}\}$ independent
conditions (see \cite[11.2]{NNI} for the explication), hence
the block $\{|{\bf m}|=j\}$ ($0\leq j\leq d-3$) contributes into $\dim \Omega_Z(D_0)=h^1(\calL_0)$
with  $\binom{j+2}{2}-\min\{k_0, \binom{j+2}{2}\}=\max\{0, \binom{j+2}{2}-k_0\}$.
In particular, $h^1(\calL_0)=\sum_{j=0}^{d-3}\max\{0, \binom{j+2}{2}-k_0\}$ and
$h^1(\calL_0)-g_{\calL_0,s}=\sum_{j=0}^{d-3-s}\max\{0, \binom{j+2}{2}-k_0\}$ ($0\leq s\leq d-2$).
Therefore,
$$d_{\calL_0,Z}(-kE^*_0)=\min_{0\leq s\leq d-2} \, \Big\{ ks +
\sum_{j=0}^{d-3-s}\max\big\{0, \textstyle{\binom{j+2}{2}}-k_0\big\}\,\Big\}.$$
However, if $\calL_0=\calO_Z(D_0)$ is not generic, then the points
 $D_0$  might fail to impose independent conditions on the corresponding linear systems, and the determination of the dimsnion of $\Omega _Z(D_0)$ can be harder.
 See \cite[11.3]{NNI} for discussion, examples and connection with the Cayley--Bacharach type theorems (cf. \cite{EH}). Those discussions with
combined with  the present section produces further examples for $d_{\calL_0,Z}(l')$ whenever $D_0$
is special (and $(X,o)$ is superisolated).

\end{example}

\section{Appendix 2. A technical lemma}\label{s:techn}

\subsection{} The next lemma is used  in the body of the article,
however,  it might have also an independent general interest.
\begin{lemma}\label{lem:techn}
Let $\tX$ be an arbitrary resolution of a normal surface singularity $(X, 0)$.
Let us  fix an arbitrary line bundle $\mathcal{L}\in {\rm Pic}(\tX)$
with $c_1(\mathcal{L}) = l' \in -S'$, an  irreducible exceptional curve $E_v$, and an integer $m>0$.

Assume that there exists a sub-vectorspace $V \subset H^0(\tX, \mathcal{L})$ with the following property:
for a generic point $p \in E_v$ there exists a section $ s \in V$
such that $s$ does not vanish along  $E_v$ and the multiplicity of the divisor of $s$ at $ p \in E_v$ is $m$. Then for any number $0 \leq k \leq m$ and a generic point $p \in E_v$ there exists a section
 $s \in  V$ such that $s$ does not vanish
along  $E_v$ and the multiplicity of the divisor of $s$ at $ p \in E_v$ is $k$.
\end{lemma}

\begin{proof}
By induction we need to prove the statement only  for  $k= m-1$.

First we fix a very large integer  $N\gg m$, and consider the
restriction
$r:H^0(\tX,{\mathcal L})\to H^0(NE_v,{\mathcal L})$. Then $r$ induces a map  from
$H^0(\tX,{\mathcal L})_{reg}:=H^0(\tX,{\mathcal L})\setminus H^0(\tX,{\mathcal L}(-E_v))$ to $H^0(NE_v,{\mathcal L})_{reg}:= H^0(NE_v,{\mathcal L})\setminus H^0((N-1)E_v,{\mathcal L}(-E_v))$.
Denote its restriction
$H^0(\tX,{\mathcal L})_{reg}\cap V\to H^0(NE_v,{\mathcal L})_{reg}\cap r(V)$ by $ r_V$.
Consider also the  natural map ${\rm div}: H^0(N E_v, \mathcal{L})_{reg} \to \eca^{l'}(N E_v)$, and
 the composition map $ {\rm div} \circ r_V = g : H^0(\tX, \mathcal{L})_{reg} \cap V \to \eca^{l'}(N E_v)$, which sends a section to its divisor restricted to the cycle $N E_v$.

Next, for any $p\in E_v^0:=E_v\setminus \cup_{u\not=v}E_u$ set $D_{m,p}
\subset \eca^{l'}(N E_v)$, the set of divisors with multiplicity $m$ at $p$. (Since
$N\gg m$ this notion is well--defined).  Set also $D_m:= \cup_p D_{m,p}$.

By the assumption, the image of $g$ intersects $D_{m,p}$ for any generic $p$.
Since $D_m$ is constructible  subvariety of $ \eca^{l'}(N E_v)$, $g^{-1}(D_m)$ is a
 nonempty constructible subset of $H^0(\tX, \mathcal{L})_{reg}\cap V$.
 Define an analytic curve $h_0:(-\epsilon,\epsilon)\to g^{-1}(D_m)$ such that its image is not
 a subset of some $g^{-1}(D_{m,p})$. Let us denote the zeros of  the section $h_0(0)$ along
 $E_v^0$ by $\{p_1, \ldots, p_r\}$. Then there exists a small neighborhood $U$ of one of the
 points $p_i$ and a restriction of $h_0$ to some smaller $(-\epsilon',\epsilon')$, such that
 for any $t\in (-\epsilon',\epsilon')$ the restriction of $h_0(t)$ to $U$ has a unique zero,
 say $p(t)$, and its multiplicity is $m$. Furthermore, $t\mapsto p(t)$, $ (-\epsilon',\epsilon')\to U\cap  E_v^0$ is not constant, hence taking further restrictions to some interval
 we can assume that $t\mapsto p(t)$ is locally invertible. Reparametrising $h_0$ by the inverse of this map, we obtain an analytic map $U\cap E_v^0\to g^{-1}(D_m)$, $t\mapsto h(t)$ such that the restriction of the section $h(t)$ to some local chart $U$ has only one zero, namely $t$, and the multiplicity of the section at $t$ is $m$. In some local coordinates $(x,y)$ of $U$
 (with $U\cap E_v=\{y=0\}$) the equation of $h(t)$ has the form  (modulo $y^N$)
\begin{equation}\label{eq:h}
h(t) = \sum_{j\geq 0 ,  i\geq 0}(x-t)^j y^i c_{j, i}(t),
\end{equation}
where by the multiplicity  condition  $c_{j, i}\equiv 0$, if $j + i < m$ and,
there is a pair $(j, i)$,
such that $j + i = m$ and $c_{j, i}(t)\not \equiv 0$. Moreover, by the non--vanishing condition
 $y\not| h(t)$, or,  $c_{j, 0}(t)\not \equiv 0$ for some $j$.

 We claim that there is a generic choice of
 $t_1, \ldots, t_r$ (for some large $r$)  of $t$--values, and a convenient choice of
 the coefficients $\{\alpha_l\}_{l=1}^r$ such that $s:=\sum_{l=1}^r \alpha_l h(t_l)$ satisfies the
 requirements.   Indeed, first we consider the Taylor expansion of $h(t)$ in variables
 $(x,y)$ at a point $(x,y)=(q,0)$ with $q$ generic (and modulo $y^N$ as usual):
 $$ \sum_{j,i} (x-q+q-t)^j y^i c_{j,i}(t)=
\sum_{j,i}  \sum_{k=0}^j (x-q)^k y^i \binom{j}{k}(q-t)^{j-k}c_{j,i}(t).$$
 The fact that $s$ at $(q,0)$ has multiplicity $\geq m-1$ transforms into a linear system
 $$\sum_{l=1}^r \alpha_l \,\Big( \, \sum_{j\geq k} \binom{j}{k} (q-t_l)^{j-k}c_{j,i}(t_l)\,\Big) =0$$
 for any $(k,i)$ with $k,i\geq 0$ and $k+i\leq m-2$.
 This linear system $LS(r,m-2)$ with unknowns $\{\alpha_l\}_{l=1}^r$ has matrix
 $M(r, m-2)$ of size
 $r\times m(m-1)/2$. If $r\gg m(m-1)/2$ then the system has a nontrivial  solution. We need to
  show that
 for a generic choice of the solutions $\{\alpha_l\}_l$ the section $s$ has multiplicity $m-1$ at $q$. Assume that this is not the case. Then the generic solution of the system $LS(r, m-2)$
 is automatically solution of $LS(r,  m-1)$ too (the last one defined similarly).
This means that ${\rm rank} M(r, m-2)={\rm rank}M(r, m-1)$ ($\dag$) for generic $\{t_l\}_l$.

The matrix  $M(r, m-1)$ has $m$ additional rows corresponding to the indexes $(k,i)$ with
$k,i\geq 0$ and $k+i=m-1$. Let us fix one of them, corresponding to
the following choice.

Now let $d$ be the minimal number, such that there exists $j, i$ such that $i \leq m-1$, $j + i = d$ and $c_{j, i}(t)$ is not identically $0$. Since by assumption
(by non--vanishing of $h(t)$ along $E_v$)
there exists certain $j\geq m$  with
$c_{j, 0}\not \equiv 0$, such a $d$ exists.
Fix $i_0$ such that  $i_0 \leq m-1$, $j_0 + i_0 = d$ and $c_{j_0, i_0}(t)\not\equiv 0$.

Then, from the additional rows of $M(r,m-1)$  we chose the one indexed by
$(m-1-i_0,i_0)$.

Consider the minor of $M(r, m-1)$ of size $m(m-1)/2+1$,
 whose last row is the row corresponding to $(m-1-i_0,i_0)$, and the other rows belong to
 $M(r, m-2)$, while the last column corresponds to the generic $t_r=t$.
Then its determinant should be zero by ($\dag$).  Expanded it by the last column gives
\begin{equation*}
 \sum_{j\geq m-1-i_0} \binom{j}{m-1-i_0} (q-t)^{j-m+1+i_0}c_{j,i_0}(t)=
 \sum_{k,i\geq 0; k+i\leq m-2}\beta_{k,i}(q)\cdot \
  \sum_{j\geq k} \binom{j}{k} (q-t)^{j-k}c_{j,i}(t)
\end{equation*}
for some holomorphic functions $\beta_{k,i}(q)$. But such an identity cannot exist.
Indeed, since $c_{j_0,i_0}\not\equiv 0$, but
$c_{j,i_0}\equiv 0$ for any $j<j_0$,
the vanishing order of  $q-t$ at the left hand side is exactly $d-m+1$,
while on the right hand side ---
 since $j\geq d-i$ (otherwise $c_{j,i}\equiv 0$)  and $k\leq m-2-i$
implies $j-k \geq d-m +2$ --- we get vanishing order $\geq d-m+ 2$.

Finally we need to show that this generic $s$ does not vanish along $E_v$.
This follows from a similar argument as above, or one can proceed as follows. For any generic $q$
consider a  section $s$ which has multiplicity $m-1$ at $(q,0)$. If it vanishes  along
$E_v$ then  $s+h(q)$ does not vanish along $E_v$ and it has multiplicity $m-1$ at $(q,0)$.
\end{proof}
\begin{remark} We claim that under the assumptions of Lemma \ref{lem:techn} the following
property also holds:
{\it For any finite set  $F\subset  E_v $
there exists a section $s\in V$ such that $s$ does not vanish
along  $E_v$, ${\rm div}(s)\cap F=\emptyset$, and at each each $p\in {\rm div}(s)\cap E_v$ the
intersection is  transversal.}
Indeed, we can use first Lemma \ref{lem:techn} for $k=1$ and then show that a generic combination of `moving' sections of multiplicity one works.
\end{remark}

\end{document}